\documentclass[12pt,espf]{amsart}
\usepackage{a4}

\input{diagrams.sty} 

\usepackage{amssymb}
\usepackage{amsmath}
\usepackage{amsthm}
\usepackage{amstext}
\usepackage{amscd}
\usepackage{latexsym}
\usepackage{graphics}
\usepackage{color}
\usepackage{url}

\theoremstyle{plain}
\newtheorem{thm}{Theorem}[section]
\newtheorem{theorem}[thm]{Theorem}

\newtheorem{lemma}[thm]{Lemma}

\newtheorem{proposition}[thm]{Proposition}

\newtheorem{corollary}[thm]{Corollary}
\newtheorem{prop}[thm]{Proposition}

\theoremstyle{definition}

\newtheorem{rmk}[thm]{Remark}

\newtheorem{defn}[thm]{Definition}
\newtheorem{ex}[thm]{Example}
\newtheorem{claim}[thm]{Claim}

\newtheorem{Definition-Proposition}[thm]{Definition-Proposition}

\renewcommand{\tilde}{\widetilde}

\newcommand{\Q}{{\mathbb Q}}
\newcommand{\R}{{\mathbb R}}

\input{xy}
\xyoption{all}

\begin{document}
\title{Existence of log canonical closures}
\author{Christopher D. Hacon}
\date{\today}
\address{Department of Mathematics \\
University of Utah\\
155 South 1400 East\\
JWB 233\\
Salt Lake City, UT 84112, USA}
\email{hacon@math.utah.edu}
\author{Chenyang Xu}
\address{
Beijing International Center of Mathematics Research\\ 5 Yiheyuan Road, Haidian District\\ 
Beijing 100871, China}
\email{cyxu@math.pku.edu.cn}
\address{Department of Mathematics \\University of Utah\\ 
155 South 1400 East \\ 
JWB 233\\
Salt Lake City, UT 84112, USA}\email{cyxu@math.utah.edu}

\begin{abstract}
Let $f:X\to U$ be a projective morphism of normal varieties and $(X,\Delta)$ a dlt pair. We prove that if there is an open set $U^0\subset U$, such that $(X,\Delta)\times_U U^0$ has a good minimal model over $U^0$ and the images of all the non-klt centers intersect $U^0$, then $(X,\Delta)$ has a good minimal model over $U$. 
As consequences we show the existence of log canonical compactifications for open log canonical pairs, and the fact that the moduli functor of stable schemes satisfies the valuative criterion for properness.
\end{abstract}
\thanks{The first author was partially supported by NSF research grant no: 0757897, the second author was partially supported by NSF research grant no: 0969495. We are grateful to O. Fujino, J. Koll\'ar and J. M$^{\rm c}$Kernan for many useful comments and suggestions. We are also in debt to J. Koll\'ar for allowing us to use the materials of \cite{Kollar10} in \S 3.}
\maketitle

\tableofcontents


\section{Introduction}
Throughout this paper, the ground field $k$ is the field of complex numbers. The main purpose of this note is to prove the following theorem.
\begin{theorem}\label{t-lccl}
 Let $f:X\to U$ be a projective morphism of normal varieties,  $\Delta$ a $\mathbb{Q}$-divisor such that  $(X,\Delta)$ is a dlt pair and $S=\lfloor \Delta \rfloor$ the non-klt locus. 
Assume that there exists an open subset  $U^0\subset U$ such that $(X^0,\Delta^0):=(X,\Delta)\times _{U} U^0$ has a good minimal model over $U^0$, and that  
any stratum of $S$ intersects $X^0$. Then $(X,\Delta)$ has a good minimal model over $U$.
\end{theorem}

There are several interesting consequences of this result.  

We begin by establishing the existence of log canonical closures of non-proper log canonical pairs.
\begin{corollary}[Existence of log canonical closure]\label{c-lcc}
Let $U^0$ be an open subset of a normal quasi-projective variety $U$, $f^0:X^0\to U^0$ a projective morphism, and $(X^0, \Delta^0)$ a log canonical pair. 
Then there exists a projective morphism $f:X\to U$ and a log canonical pair $(X,\Delta)$ such that  $X^0= X\times _ U U^0$ is an open subset and $\Delta^0=\Delta|_{X^0}$. 
\end{corollary}

Our next application establishes the existence of compactifications of log canonical morphisms. 
\begin{defn}[{\cite[7.1]{KM98}}]\label{d-lc} Let $X$ be a normal variety, $f : X \to U$ a dominant morphism to a smooth curve $U$, 
$\Delta$ an effective $\mathbb{Q}$-divisor such that $K_X +\Delta$ is $\mathbb{Q}$-Cartier. We say that $f$ is a {\it log canonical morphism}, or an {\it lc morphism}, if $(X, \Delta + X_p)$ is lc for all closed points $p \in U$, where $X_p$ is the fiber over $p$.
\end{defn}

As a corollary to \eqref{t-lccl}, we give an affirmative answer to the following conjecture due to Koll\'ar-Kov\'acs.
\begin{corollary}[{\cite[Conjecture 7.16]{KK10}}]\label{c-lcm}Let $U$ be a smooth curve.
Let $f^0 : X^0 \to U$ be an affine finite type lc morphism. Then there exists a finite dominating base change morphism $\theta : \tilde{U} \to U$ and a projective lc morphism $ f: X \to \tilde{U}$ such that $X^0\times_U \tilde {U}\subset X$  and $f|_{X^0\times_U \tilde {U}} =f^0\times_U \theta$. 
\end{corollary}

Next we prove the following statement which implies the properness of the moduli functor of stable schemes (cf. \S \ref{s-func}).  

\begin{corollary}\label{c-pm}
Let $f^0:X^0\to U^0$ be a projective morphism, $(X^0,\Delta^0)$ a log canonical pair, $U$ the germ of a smooth curve, $p\in U$ a closed point and $U^0=U\setminus \{ p\}$. If $K_{X^0}+\Delta^0$ is $f^0$-ample, 
then there is a finite dominating base change $\theta:\tilde{U}\to U$, a log canonical pair $(X,\Delta)$ 
and a projective lc morphism $(X,\Delta)\to \tilde{U}$ such that $K_X+\Delta$ is ample over $\tilde{U}$ and the restriction of $(X,\Delta)$ to the pre-image  $\theta^{-1}(U^0)$ is isomorphic to $ (X^0,\Delta^0)\times_{U}\tilde{U} $. 
\end{corollary}

The techniques developed for proving \eqref{t-lccl} can also be used to verify the following statement, which is conjectured by J. Koll\'ar as a tool to study the geometry of $\epsilon$-lc centers.

\begin{theorem}[{\cite[Conjecture 4.11]{Kollar11}}]\label{c-tri}
Let $f : X \to U$ be a projective morphism between normal varieties, $\Delta'$ and $\Delta''$ effective $\mathbb{Q}$-divisors on $X$ such that $(X, \Delta' + \Delta'')$ is a $\mathbb{ Q}$-factorial lc pair, $(X,\Delta'')$ is dlt and $K_X + \Delta' + \Delta'' \sim_{\mathbb{Q},U} 0$. Then the $(K_X+\Delta'')$-MMP with scaling over $U$ terminates with either a Mori fibration or a $\mathbb{Q}$-factorial good minimal model. \end{theorem}

\begin{rmk}Note that several of the above results were proven independently by C. Birkar (cf. \cite{Birkar11}) under some mild additional assumptions. In particular, compare \cite[1.3, 1.9]{Birkar11}  to \eqref{t-lccl} and \eqref{c-tri}. Moreover, as explained in \cite[1.6]{Birkar11}, \eqref{c-tri} immediately implies the existence of log canonical flips cf. \eqref{c-flips} below.\end{rmk} 
\begin{corollary}\label{c-flips} Let $f:X\to Z$ be a flipping contraction for  a log canonical pair $(X,\Delta )$. Then the flip
of $f$ exists.
\end{corollary}

We will now briefly sketch the idea of the proof of \eqref{t-lccl} and give an outline of this paper. 
The proof is by induction on the dimension and is divided into three main steps which are given in Sections 3-5. 
The goal is to show that the dlt pair $(X,\Delta )$ admits a good minimal model over $U$. Because of the inductive structure of the proof, it is necessary to establish an analogous  statement for the non-klt locus $S=\lfloor \Delta \rfloor$ or more precisely for the sdlt pair $(S,\Delta _S)$ where $K_S+\Delta _S
:=(K_X+\Delta )|_S$ is defined by adjunction.
Note that $S$ is not normal, however we may assume by induction on the dimension, that the required statement holds on each irreducible component $S_i$ of $S$. We then use Koll\'ar's gluing theory to deduce the result on $S$. This step is achieved in Section 3.

In Section 4, we observe that the result in Section 3 and a result of Fujino (cf. \cite{Fujino05}), which generalizes the main theorem of Kawamata in \cite{Kawamata85}, immediately imply that any minimal model of $(X,\Delta)$ is indeed a good model.

Therefore, the proof of \eqref{t-lccl} is reduced to proving 
the existence of a minimal model for the dlt pair $(X,\Delta )$.  This is the most technical step. The difficulty is two-fold. First, since the pair is not 
of log general type, we must work with the Iitaka fibration and apply 
Kawamata's canonical bundle formula. This introduces several technical difficulties. 
Second, as already apparent in \cite{BCHM10}, it is hard to directly show the termination of flips. 
Instead we follow some ideas developed in \cite[Section 5]{BCHM10}: we first see that it suffices to find a {\it neutral model} (see \cite[5.3]{BCHM10} for the definition) which contracts the `right' components. 
Then to construct such a neutral model, we use Shokurov's idea which reduces the problem to a special termination question (cf. \cite{Shokurov92}), so that we can apply induction on the dimension.
 
In Section 6, we prove the above mentioned corollaries to \eqref{t-lccl}, and in the last section, we discuss (following ideas due to V. Alexeev, J. Koll\'ar, N. Shepherd-Barron and others) the relationship between the results of this paper and the question of the existence of a natural compactification of the moduli space of canonically polarized varieties. 

\section{Preliminaries}
We will follow the terminology from \cite{KM98}.  We will also need the definition of certain singularities of semi-normal pairs. 
Let $X$ be a semi-normal variety which satisfies Serre's condition $S_2$ and $\Delta$ be a $\Q$-divisor on $X$, such that $K_X+\Delta$ is $\mathbb{Q}$-Cartier. 
Let $n:X^n\to X$ be the normalization of $X$ and write $n^*(K_X+\Delta)=K_{X^n}+\Delta^n +\Gamma$, where $\Gamma$ is the reduced double locus. 
We say that $(X,\Delta)$ is {\it semi-log canonical} or {\it slc} if $(X^n,\Delta^n+\Gamma)$ is log canonical and $(X,\Delta)$ is {\it divisorial semi-log-terminal} or {\it dslt} if $(X^n,\Delta^n+\Gamma)$ is dlt. 
Note that if $(X,\Delta)$ is dlt and $B$ is a union of components of $\lfloor \Delta \rfloor$, then $(B,{\rm Diff}^*_{B}(\Delta))$ is $dslt$, where $K_B+{\rm Diff}^*_{B}(\Delta)=(K_X+\Delta)|_B$.

Given a $dlt$ pair $(X,\Delta)$ with non-klt locus $S=\lfloor \Delta \rfloor$, write $S=\sum^{m}_{i=1} S_i$ where each $S_i$ is a prime divisor. We say that $Z\subset  S$ is
{\it a stratum of $S$} if $Z$ is a non-empty irreducible component of $S_I=\cap _{i\in I}S_i$, for some subset $I\subset \{1,2,\ldots ,m\}$. 

Let $g:X\to Y$ be a proper morphism between two normal varieties. We say that $g$ is an {\it algebraic fibration} if $g_*\mathcal{O}_X=\mathcal{O}_Y$. If $f:X\to U$ is an  algebraic fibration,  and $\mathcal F$ is a coherent sheaf on $X$, then $\mathcal F$ is {\it generated over} $U$ (or $f$-{\it generated}) if the homomorphism $f^*f_*\mathcal F\to \mathcal F$ is surjective. 
Notice that if $g:U\to V$ is another algebraic fibration, then $\mathcal F$ is generated over $V$ implies that $\mathcal F$ is generated over $U$ (in fact since the composition $(g\circ f)^*(g\circ f)_*\mathcal F=f^*(g^*g_*(f_*\mathcal F))\to  f^*f_*\mathcal F\to \mathcal F$ is surjective, so is the homomorphism $ f^*f_*\mathcal F\to \mathcal F$).

Let $\phi:X\dasharrow Y$ be a proper birational contraction of normal quasi-projective varieties (so that in particular $\phi ^{-1}$ contracts no divisors).
If $D$ is a $\Q$-Cartier divisor on $X$ such that $D':=\phi _*D$ is $\Q$-Cartier then we say that $\phi$ is $D$-{\it non-positive} (resp. $D$-{\it negative}) if for a common resolution $p:W\to X$ and $q:W\to Y$, we have $p^*D=q^*D'+E$ where $E\geq 0$ and $p_*E$ is $\phi$-exceptional (respectively the support of $p_*E$  equals the set of $\phi$-exceptional divisors).
Suppose that $f:X\to U$ and $f_M:X_M\to U$ are projective morphisms, $\phi:X\dasharrow X_M$ is a birational contraction  and $(X,\Delta )$ and $(X_M,\Delta _M)$ 
are log canonical pairs, klt pairs,  or dlt pairs where $\Delta_M=\phi_*\Delta$. 
If $a(E;X,\Delta )>a(E;X_M,\Delta _M)$  for all $\phi$-exceptional divisors $E\subset X$, $X_M$ is $\Q$-factorial and $K_{X_M}+\Delta_M$ is nef  over $U$, then we say that 
$\phi:X\dasharrow X_M$ is  {\it a minimal model}. If instead $a(E;X,\Delta )\ge a(E;X_M,\Delta _M)$ for all divisors $E$  and $K_{X_M}+\Delta_M$ is nef, then we call $X_M$ a  {\it weak log canonical model} of  $K_X+\Delta$ over $U$.  

We say $K_{X_M}+\Delta_M$  is {\it semi-ample} over $U$ if there exists a morphism $\psi : X_M\to Z$ between normal varieties projective over $U$ such that  $K_{X_M}+\Delta_M\sim_{\mathbb{Q}}\psi ^* A$ for some $\Q$-divisor $A$ on $Z$ which is ample over $U$.
When $\psi_*\mathcal{O}_X=\mathcal{O}_Z$,  $K_{X_M}+\Delta_M$  is semi-ample over $U$ if and only if there exists an integer $m>0$ such that $\mathcal O_{X_M}(m(K_{X_M}+\Delta_M))$ is generated over $U$.
Note that in this case $$R(X_M/U,K_{X_M}+\Delta_M):=\bigoplus_{m\geq 0}f_*\mathcal O_{X_M}(m(K_{X_M}+\Delta_M))$$ is a finitely generated $\mathcal{O}_U$-algebra and  $Z={\rm Proj}R(X_M/U,K_{X_M}+\Delta_M)$. 
Recall that for any $\mathbb{Q}$-divisor $D$ on $X$,  the sheaf $f_*\mathcal{O}_X(D)$ is defined to be $f_*\mathcal{O}_X(\lfloor D \rfloor)$. 
A minimal model $\phi:X\dasharrow X_M$ is called {\it a good minimal model} (or {\it a good model} for short) if $K_{X_M}+\Delta_M$ is semi-ample. 
If $K_{X_M}+\Delta_M$ is semi-ample and big over $U$, then we let  
$X_{LC}={\rm Proj}R(X/U,K_X+\Delta)$ be the {\it log canonical model } of $(X,\Delta)$ over $U$. More generally, we say that a birational contraction $g : X \dasharrow Y$ over $U$ is a {\it semi-ample model} of a $\Q$-Cartier divisor $D$ over $U$ if $g$ is $D$-non-positive, $Y$ is normal and projective over $U$ and $H = g_*D$ is semi-ample over $U$. 

Let $X$ be a projective variety, $B$ a big $\mathbb{R}$-divisor on $X$ and $C$ a prime divisor on $X$. Then we have
$$\sigma_{C}(B) := \inf\{{\rm mult}_C(B')|B' \sim_{\mathbb{R}} B, B' \ge 0\}. $$
Note that $\sigma_{C}$ is a continuous function on the cone of big divisors (cf. \cite[III.1.7]{Nakayama04}). Now let $D$ be any pseudo-effective $\mathbb{R}$-divisor and let $A$ be any ample $\mathbb{Q}$-divisor. 
Define $$\sigma_{C}(D) = \lim_{\epsilon\to 0} \sigma_C(D + \epsilon A).$$Then $\sigma_C(D)$ exists and is independent of the choice of $A$. There are only finitely many prime divisors $C$ such that $\sigma_C(D) > 0$ and the $\mathbb{R}$-divisor 
$$N_{\sigma}(D) = \sum \sigma_C(D)\cdot C$$ 
is determined by the numerical equivalence class of $D$ (cf. \cite[III.1]{Nakayama04} or \cite[3.3]{BCHM10}).  For the relative case, let $f:X\to U$ be a projective morphism of normal varieties, we can similarly define $\sigma_\Gamma(D/U)$ and $N_{\sigma}(D/U)$ for $f$-pseudo-effective divisors as in \cite[III.4]{Nakayama04}. We note that it is not known that $\sigma _\Gamma (D/U) <+\infty$ always holds (however see \cite[III.4.3]{Nakayama04} for some cases in which the answer is known).
In this paper we will only consider the case in which $ D\sim _{\Q , f}\lambda (K_X+\Delta)$ where $\lambda >0$ and $(X,\Delta )$ is a dlt pair. In this case $\sigma _\Gamma (D/U) <+\infty$ always holds by \cite{BCHM10}.

Let $(X,\Delta)$ be a projective log pair. The group ${\rm Bir}(X,\Delta)$ consists of all birational self maps $\phi:X\dasharrow X$, such that if we let $X^d$ be a resolution of the indeterminacy
\begin{diagram}
&& X^d&&\\
& \ldTo^p & & \rdTo^q\\
X& &\rDashto^{\phi}&&X,
\end{diagram}
then $p^*(K_X+\Delta)=q^*(K_X+\Delta)$ (cf. \cite{Fujino00}).
 For any positive integer $m$,
We call the homomorphism
$$\rho_{m}:{\rm Bir}(X,\Delta)\to {\rm Aut}(H^0(X, \mathcal O_{X}(m(K_{X}+\Delta))))$$
the {\it {\bf B}-representation}.  

\subsection{Canonical bundle formula}
In this subsection, we will give a version of the canonical bundle formula that follows from the work of Kawamata, Fujino-Mori and Ambro (cf. \cite{Kawamata98}, \cite{FM00}, \cite{Ambro04} and \cite{Kollar07}). 
\begin{theorem}\label{t-can}
Let $(X,\Delta)$ be a dlt pair and $f:X\to U$ a projective morphism over 
a normal variety $U$. Then there exists a commutative diagram of projective morphisms
\begin{diagram}
X'&\rTo^{\mu} & X\\
\dTo^{h} & &\dTo^f\\
Y&\rTo^g& U
\end{diagram}
with the following properties
\begin{enumerate}
\item $\mu$ is a birational morphism, $h$ is an equidimensional algebraic fibration, $X'$ has only $\mathbb{Q}$-factorial toroidal singularities and $Y$ is smooth; 
\item there exists a $\Q$-divisor $\Delta'$ on $X'$ with coefficients $\leq 1$, such that $(X',{\rm Supp}(\Delta'))$ is quasi-smooth (i.e., $(X',{\rm Supp}(\Delta'))$ is toriodal and $X'$ is $\mathbb{Q}$-factorial),
and $$\mu_*\mathcal{O}_{X'}(m(K_{X'}+\Delta'))\cong \mathcal{O}_X(m(K_X+\Delta)), \ \forall m\in \mathbb{N};$$
\item there exist  a $\Q$-divisor $B$ and a $\mathbb{Q}$-line bundle $J$ such that $B$ is effective, $K_Y+B+J$ is big over $U$, $J$ is $g$-nef and a positive integer $d$  such that  we can write
$$h^*(K_Y+B+J)+R\sim_{\mathbb{Q}}K_{X'}+\Delta',$$
where  $R\ge 0$ and $h_*\mathcal{O}_{X'}(dm(K_{X'}+\Delta'))\cong \mathcal{O}_Y(dm(K_{Y}+J+B))$, for all $m\in \mathbb{N}$;
\item $(Y,{\rm Supp (B)})$ log smooth, the coefficients of $B$ are $\leq 1$ and each component of $\lfloor B\rfloor$ is dominated by a vertical component of $\lfloor \Delta \rfloor$. 

\end{enumerate}
\end{theorem}
\begin{proof}
We may choose a birational projective morphism $\mu:X'\to X$, such that there exists a projective morphism $h:X'\to Y $ of smooth projective varieties over $U$ and the restriction $h_{\eta}:X'_\eta \to Y_\eta $  over the generic point $\eta$ of $U$ is birational to the the Iitaka fibration of $X_{\eta}$ over $\eta$.
By the weak semi-stable reduction theorem of Abramovich and Karu (cf. \cite{AK00}), we can assume that, $h:(X',D')\to (Y,D_Y)$ is an equidimensional  toroidal morphism for some divisors $D'$ on $X'$ and $D_Y$ on $Y$ where $(X',D')$ is quasi-smooth, $Y$ is smooth and  $\mu^{-1}(\Delta\cup {\rm Sing}(X))\subset D'$ (see \cite[2]{Kawamata10}). Therefore, if we write $\mu^*(K_X+\Delta)+F=K_{X'}+\Delta'$, where $F$ and $\Delta'$ are effective with no common components, then ${\rm Supp}(\Delta')\subset D'$. Clearly $(X',\Delta')$ satisfies (1) and (2).

 It follows from the proof of \cite[4.5]{FM00} that there exists a $\mathbb{Q}$-divisor $R$ on $X'$ such that 
\begin{enumerate}
\item[$\bullet$] $K_{X'}+\Delta'\sim_{\mathbb{Q}}h^*(K_Y+B+J)+R$, where $B$ is {\it the  boundary part} and $J$ is {\it the moduli part}.
\item[$\bullet$] $R $ is effective and $h_*\mathcal{O}_{X'}(iR)\cong \mathcal{O}_{Y}$ for all $i\ge 0$. 
\end{enumerate}
To see that $R$ is effective, write  $R=R^{>0}-R^{<0}$ where $R^{>0}$ and $R^{<0}$ are effective with no common components and recall that by \cite[4.5(ii)]{FM00}, 
${\rm codim}(h(R^{<0})\subset Y')\ge 2.$ Since $h$ is equidimensional, we have $R^{<0}=0$.

If we write $B=\sum t_PP$, where $P$ are codimension 1 points on $Y$, then $t_P=1-s_P$, where $s_P$ is the log canonical threshold of $h^{-1}(P)$ with respect to $(X,\Delta'-R)$ over the generic point of  $P$ (cf. \cite[4.3]{FM00}, \cite[8.5.1]{Kollar07}).
 Since $h_*\mathcal{O}_{X'}(iR)\cong \mathcal{O}_{Y}$, it follows that
${\rm Supp}(R)$ does not contain all the components of $h^{-1}(P)$ 
which dominate $P$. Thus we have $s_P\le1$ and so $t_P\ge 0$. On the other hand, $R$ is effective over the generic point of $P$. Therefore, $s_P$ is not less than the log canonical threshold of $h^{-1}(P)$ with respect to $(X',\Delta ')$ over the generic point of $P$. We also have ${\rm Supp}(B)\subset D_Y $. Thus we obtain (4).

To verify that $J$ is nef over $U$, we first remark that by \cite[1(4)]{Kawamata10}, the $\mathbb{Q}$-line bundle $J$ computed from $X'\to Y$ commutes 
with any pull back in the sense of \cite[8.4.9(3)]{Kollar07}. Since $J$ is defined by a variation of mixed Hodge structures (cf. \cite[8.4.5(7)]{Kollar07}), 
its $g$-nefness follows from \cite{Fujino04} or \cite[25]{Kawamata10}.

\end{proof}

We need the following simple lemma.
\begin{lemma}\label{l-proj}
Let $h:X\to Y$ and $h':X'\to Y'$ be a projective algebraic fiber space over $U$. 
Let $\mu: X\dashrightarrow X'$ and $\eta: Y \dashrightarrow Y'$ be birational contractions. Let $D$ be a $\mathbb{Q}$-Cartier divisor on $X$ such that $D\sim_{\mathbb{Q},U} h^*E$ for some $\mathbb{Q}$-Cartier  $\mathbb{Q}$-divisor $E$ on $Y$. Assume that $\eta_*E$ is $\mathbb{Q}$-Cartier  on $Y'$.

 If $\mu_*D\sim_{\mathbb{Q},Y'} 0$ 
 then $\mu_* D \sim_{\mathbb{Q},U}h'^*(\eta_*E). $
\end{lemma}
\begin{proof}Replacing $E$ by a relatively $\mathbb{Q}$-linearly equivalent $\mathbb{Q}$-divisor, we can assume that $D= h^*E$ (cf. \cite[8.3.5]{Kollar07}).

If $\mu_*D\sim_{\mathbb{Q},Y'} 0$,   then $F:=\mu_* D -h'^*(\eta_*E)$ is a divisor on $X'$ whose image in $Y'$ is of codimension $\geq 2$. By our assumption, it follows that $F\sim_{\mathbb{Q},Y'}0$ and so $F=0$ (cf. \cite[8.3.5]{Kollar07}).
\end{proof}

\subsection{Minimal Models}
In this subsection, we will collect results on the existence of good minimal models. A large part of them are already known. 

\begin{theorem}\label{t-bchm} Let $f:X\to U$ be a projective morphism, $(X,\Delta )$ a $\Q$-factorial dlt pair, $S=\lfloor \Delta \rfloor$ the non-klt locus. Assume that 
either 
\begin{enumerate}
\item $\Delta $ is big over $U$ and no strata of $S$ is contained in ${\bf B}_+(\Delta /U )$,
or 
\item $K_X+\Delta $ is  big over $U$ and no strata of $S$ is contained in ${\bf B}_+(K_X+\Delta /U)$, or
\item  $K_X+\Delta $ is not pseudo-effective over $U$.
\end{enumerate}
Then the $(K_X+\Delta) $-minimal model program with scaling of an ample divisor over $U$ terminates with either a good minimal model or a Fano contraction.
\end{theorem}
\begin{proof} This is an immediate consequence of 
\cite[1.3.3, 1.4.2]{BCHM10} and their proofs.
\end{proof}
\begin{lemma}\label{l-gm} Let $f:X\to U$ be a projective morphism, $(X,\Delta )$ a dlt pair and $\phi :X\dasharrow X_M$ and $\phi ':X\dasharrow X_M'$ be minimal models for $K_X+\Delta$ over $U$. Then 
\begin{enumerate}
\item the set of $\phi$-exceptional divisors concides with the set of divisors contained in ${\bf B}_-(K_X+\Delta /U)$ and if $\phi$ is a good minimal model for  $K_X+\Delta$ over $U$, then this set also coincides with the set of divisors contained in ${\bf B}(K_X+\Delta /U)$,
\item $X'_M\dasharrow X_M$ is an isomorphism in codimension $1$ such that
$a(E;X_M,\phi _* \Delta )=a(E;X_M',\phi '_* \Delta )$ for any divisor $E$ over $X$, and
\item if $\phi$ is a good minimal model of $K_X+\Delta$  over $U$, then so is $\phi ' $.
\end{enumerate}
\end{lemma}
\begin{proof} Let $p:Y\to X$ and $q:Y\to X_M$ be a common resolution.
Since $\phi$ is $(K_X+\Delta)$-negative, we have that $p^*(K_X+\Delta)=q^*(K_{X_M}+\phi _* \Delta )+E$ where $E$ is effective, $q$-exceptional and the support of $p_*E$ is the set of $\phi$-exceptional divisors.
By \cite[5.14]{Nakayama04}, we have $N_\sigma (p^*(K_{X}+\Delta )/U)=E$.
By \cite[5.15]{Nakayama04}, we have $N_\sigma (K_{X}+\Delta /U)=p_* E$.
This proves (1).

It follows from (1) that $X'_M\dasharrow X_M$ is an isomorphism in codimension $1$. By the Negativity Lemma (cf. \cite[3.6.2]{BCHM10}), we have that
$a(E;X_M,\phi _* \Delta )=a(E;X_M',\phi '_* \Delta )$ for any divisor $E$ over $X$. Thus (2) holds.

Let $p:Y\to X'_M$ and $q:Y\to X_M$ be a common resolution. By (2), we have that $p^*(K_{X_M'}+\phi '_*\Delta)=q^*(K_{X_M}+\phi _* \Delta )$, and so both of these are semiample over $U$ if one of them is. (3) follows immediately.
\end{proof}

\begin{lemma}\label{l-1} Let $f:X\to U$ be a projective morphism, $(X,\Delta )$ a dlt pair and $\phi :X\dasharrow X'$ a birational contraction over $U$ such that $X'$ is $\Q$-factorial and the support of  ${\bf Fix}(K_X+\Delta /U)$ (i.e., the divisorial components of ${\bf B}(K_X+\Delta/U)$) equals the set of $\phi$-exceptional divisors. 
If $ (X,\Delta )$ has a good minimal model over $U$ say $\psi: X\dasharrow X_M$ and $K_{X'}+\phi _*\Delta $ is nef over $U$, then $\phi$ is a  minimal model of $(X,\Delta)$ over $U$.\end{lemma}
\begin{proof} Since $\psi:X\dasharrow X_M$ is a good minimal model, by \eqref{l-gm} the support of
${\bf Fix}(K_X+\Delta /U)$ equals the set of $\psi$-exceptional divisors and so $X'\dasharrow X_M$ is an isomorphism in codimension $1$. But then, by the Negativity Lemma (cf. \cite[3.6.2]{BCHM10}), it follows that 
$a(E,X',\phi _* \Delta )=a(E,X_M,\psi _* \Delta )$ for all divisors $E$ over $X'$. Thus $\phi$ is $(K_X+\Delta)$-negative.
\end{proof}

\begin{lemma}\label{l-rel} 
Let $f:X\to U$ be a projective morphism, $(X,\Delta )$ and  $(X,\Delta ' )$ be pairs such that $(X,\Delta _t:=(1-t)\Delta +t\Delta' )$ is a dlt pair for any $0\leq t\leq 1$.
Assume that $K_X+\Delta$ is semi-ample over $U$ and $g:X\to Z:={\rm Proj }R(X/U,K_X+\Delta )$ is the corresponding morphism.
If $(X,\Delta ')$  admits a good minimal model $h:X\dasharrow X_M$ over $Z$, 
then
$h$ is a minimal model of $(X,\Delta _t)$  over $U$ for all $0<t \ll 1$.\end{lemma}
\begin{proof} We follow ideas from \cite{Shokurov96}. By our assumption, there is a $\Q$-divisor $H$ on $Z$ which is ample over $U$ such that $K_X+\Delta \sim _{\Q , U} g^*H$. Let $m>0$ be an integer such that $mH$ is Cartier. We claim that $K_{X_M}+h_*\Delta _t$ is nef over $U$ for any $0\leq t\leq \frac 1{1+2m\dim X}$.

Suppose not, then $K_{X_M}+h_*\Delta _t$ is not nef over $U$, and there is a
 $(K_{X_M}+h_*\Delta _t)$-negative extremal ray $R$ in $\overline{\rm NE}(X_M/U)$.
Note that as $K_{X_M}+h_*\Delta$ is nef over $U$,  $R$ is also a $(K_{X_M}+h_*\Delta ')$-negative extremal ray and so
it is spanned by a curve $\Sigma$ such that $0<-(K_{X_M}+h_*\Delta ')\cdot \Sigma\leq 2\dim X$ by a result of Kawamata (cf. \cite{Kawamata91} or \cite[3.8.2]{BCHM10}). Moreover, as $K_{X_M}+h_*\Delta_t$ is nef over $Z$, we have that $g_{M*}\Sigma \ne 0$ where $g_M:X_M\to Z$. Thus
$$0>(K_{X_M}+h_*\Delta _t)\cdot \Sigma=t (K_{X_M}+h_*\Delta ')\cdot \Sigma+(1-t)(K_{X_M}+h_*\Delta)\cdot \Sigma$$
$$\geq -2t\dim X +(1-t)H\cdot g_{M*}\Sigma\geq -2t\dim X +(1-t)\frac 1 m\geq 0.$$
This is impossible and so $K_{X_M}+h_*\Delta _t$ is nef over $U$ for all $0<t \leq \frac 1{1+2m\dim X}$.

Since $h$ is $(K_X+\Delta')$-negative and $K_X+\Delta\sim _{\Q,Z}0$,  $h$ is $(K_X+\Delta_t)$-negative for any $0<t\leq 1$, we conclude that $X_M$ is a minimal model of $(X,\Delta_t)$ over $U$.
\end{proof}

We will need the following result.
\begin{lemma}\label{dltterm}
Let $(X,\Delta )$ be a $\Q$-factorial dlt pair, $f:X\to U$ a projective morphism and $A$  an $f$-ample $\Q$-divisor.
Then the following are equivalent.
\begin{enumerate}
\item $R(X/U;K_X+\Delta,K_X+\Delta +A):=\bigoplus _{n_1,n_2\in \mathbb N^2}f_*\mathcal O _X(n_1(K_X+\Delta)+n_2(K_X+\Delta +A))$ is finitely generated,
\item $(X,\Delta )$ has a good minimal model over $U$ and the $(K_X+\Delta)$-MMP  over $U$ with scaling of $A$ terminates. 
\end{enumerate}\end{lemma}
\begin{proof} By \cite[6.8, 7.1]{CL10} and its proof (1) implies (2).
Conversely if the $(K_X+\Delta)$-MMP  over $U$ with scaling of $A$ terminates,
then there exists a rational number $t_0>0$ and
 birational contraction over $U$ say $\phi :X\dasharrow X'$ which is a $K_X+\Delta +tA$ good minimal model over $U$ for all $0\leq t\leq t_0$.
Thus 
$$R(X/U;K_X+\Delta,K_X+\Delta +A):=\bigoplus _{n_1,n_2\in \mathbb N^2}f_*\mathcal O _X(n_1(K_X+\Delta)+n_2(K_X+\Delta +A))$$ 
is finitely generated by \cite[Theorem E]{BCHM10} and \eqref{l-cox} below.
\end{proof}
\begin{lemma}\label{l-cox} Let $f:X\to U$ be a projective morphism and $A$ and $B$ be semi-ample divisors over $U$.
Then the ring $R(X/U; A,B)$ is finitely generated.\end{lemma}
\begin{proof} Well known.\end{proof}

Next, we recall the following result of Kawamata (cf. \cite[Corollary 5]{Kawamata11}). For the reader's convenience, we include a proof.
\begin{corollary}\label{dltmmp1}Let  $f:X\to U$ be a projective morphism and $(X,\Delta )$ a $\Q$-factorial dlt pair with a good minimal model over $U$.
Then any $(K_X+\Delta)$-minimal model program over $U$ with scaling of an ample divisor terminates.
\end{corollary}
\begin{proof}
Let $\phi:X\dasharrow X_M$ be the good minimal model of $(X,\Delta)$ over $U$ and $X'$ a birational model of $X$ with two proper birational morphisms $\mu:X'\to X$ and $q:X'\to X_M$. 
We have $\mu ^* (K_X+\Delta)=q^*(K_{X_M}+\Delta _M)+E$ where $\Delta _M=\phi _* \Delta$ and $E$ is effective, $q$-exceptional and $\mu_*E$ is supported on
the $\phi$-exceptional divisors.
Write $\mu^*(K_{X}+\Delta)+F=K_{X'}+\Delta'$, where $F$ and $\Delta'$ are effective and have no common components. 
Then $K_{X'}+\Delta '=q^*(K_{X_M}+\Delta_M)+E+F$.
Since $E+F$ is $q$-exceptional, one of its components is contained in ${\bf B}_{-}(K_{X'}+\Delta'/X_M)$. Running a $(K_{X'}+\Delta')$-MMP over $X_M$ with scaling of an ample divisor, we contract the divisorial components of 
${\bf B}_{-}(K_{X'}+\Delta'/X_M)$ and hence some component of the support of $E+F$. We obtain a rational map 
$\phi':X'\dasharrow X'_M$ such that $\phi '_*(F+E)=0$ and hence $\mu_M^*(K_{X_M}+\Delta_M)=K_{X'_M}+\Delta_M'$ where $\Delta_M'=\phi '_* \Delta '$ and $\mu _M:X'_M\to X_M$ is the induced morphism.
Thus $\phi'$ is a good minimal model
for $(X',\Delta')$ over $X_M$. 
In particular, $\phi '$ is also a good minimal model of $(X',\Delta')$ over $U$.

Let $X'_M\to Y={\rm Proj}R(X'/U;K_{X'}+\Delta ')$ be the induced morphism.
Let $A$ be a general very ample $\Q$-divisor on $X$, $A'=\mu^*A$  and $\Delta _1\sim _{\Q,U}\Delta +A$ be $\Q$-divisor such that $(X,\Delta _1)$ is klt and $\Delta _1\geq \epsilon A$ for some $0<\epsilon \ll 1$. 
Let $\Delta _t=(1-t)\Delta +t\Delta _1$ then $(X,\Delta _t)$ is klt for $0<t\leq 1$. We write 
$$K_{X'}+\Delta '_t-F_t=\mu ^*(K_X+\Delta _t),$$ where $\Delta' _t$ and $F_t$ are effective with no common component. There exists a rational number $0<t_0 \ll1$, such that $ {\rm Supp} (\Delta')\subset {\rm Supp}(\Delta'_{t_0})$ and $ {\rm Supp} (F)\subset {\rm Supp}(F_{t_0})$. For $0\le t \le t_0$, we have $\Delta'_t=(1-\frac{t}{t_0})\Delta'+\frac{t}{t_0}\Delta'_{t_0}$.

Note that $(X'_M,\Delta '_M=\phi '_*\Delta ')$ is dlt and $\phi '$ is an isomorphism at the general point of any strata of $\lfloor \phi '_* \Delta '\rfloor$ (cf. \cite[3.10.11]{BCHM10}) and hence we can assume that for any $0<t\le t_0\ll 1$ we have
$$(X'_M,\phi' _* \Delta' _t)\mbox{ is klt }\ \qquad \mbox{and} \qquad (X'_M,\phi '_*(\Delta '+tA'))\mbox{ is dlt}.$$ Note that $\phi '_*\Delta '_{t}$ is big over $U$ and hence over $Y$ for any $t>0$.
Let $X'_M\dasharrow X''_M$ be a good minimal model  of $K_{X'_M}+\phi '_*\Delta '_{t}$ over $Y$ (cf. \eqref{t-bchm}). 
By \eqref{l-rel}, $X'_M\dasharrow X''_M$ is a good minimal model of 
$K_{X'_M}+\phi '_*\Delta '_{\alpha }$ over $U$ for any $0<\alpha \ll t_0$. Applying \eqref{l-cox} to $X''_M$, we conclude
that $R(X'_M/U; K_{X'_M}+\phi '_*\Delta',  K_{X'_M}+\phi '_* \Delta '_\alpha )$ is finitely generated for $0<\alpha \ll t_0$.
Since any $(K_{X'}+\Delta ')$-flip or divisorial contraction is also a 
$(K_{X'}+\Delta'_\alpha) $-flip or divisorial contraction for any $0\leq \alpha \ll t_0$, we can assume that $\phi'$ is also $(K_{X'}+\Delta'_\alpha)$-non-positive and
it follows that $R(X'/U; K_{X'}+\Delta',  K_{X'}+\Delta'_\alpha )$ is finitely generated for $0<\alpha\ll t_0$.
Since $$K_{X'}+\Delta'_t \sim _{\Q , U}f^*(K_X+\Delta +t A)+ F_{t}\qquad \forall \ 0\leq t \leq t_0$$ and $F_t$ is $\mu$-exceptional, $R(X/U; K_{X}+\Delta,  K_{X}+\Delta+\alpha A)$ is finitely generated for $0<\alpha \ll t_0$.
By \eqref{dltterm}, the $(K_X+\Delta)$-MMP over $U$ with scaling of $A$ terminates and this concludes the proof.

\end{proof}
\begin{lemma}\label{l-birmm} Let $f:X\to U$ be a projective morphism, $(X,\Delta )$ a dlt pair and $\mu :X'\to X$  a proper birational morphism.
We write $K_{X'}+\Delta '=\mu ^*(K_X+\Delta )+F$ where $\Delta '$ and $F$ are effective with no common components.

Then $(X,\Delta)$ has a good minimal model over $U$ if and only if  
$(X',\Delta')$ has a good minimal model over $U$.
\end{lemma}
\begin{proof} Suppose that $(X,\Delta )$ has a good minimal model over $U$ say $\phi :X\dasharrow X_M$, then $(X_M,\Delta _M=\phi _* \Delta )$ is dlt.
By \eqref{dltmmp1}, we may assume that $\phi$ is the output of a minimal model program with scaling.
Let $\mathcal E$ be the union of the $\mu$-exceptional divisors $E\subset X'$ such that $a(E;X,\Delta )\leq 0$ and whose center $V$ is not contained in ${\bf B}(K_X+\Delta /U)$. In particular no component of $F$ is contained in $\mathcal E$.
Since $\phi$ is an isomorphism at the generic point of $V$, by \cite[1.4.3]{BCHM10}, there exists a proper birational morphism of normal varieties $\mu _M:X'_M\to X_M$ 
whose exceptional divisors correspond to the divisors in $\mathcal E$.
Thus, $\phi ':X'\dasharrow X'_M$ is a birational contraction such that $\phi '_*(F)=0$ and hence
$\phi '_*(K_{X'}+\Delta ')=\mu _M^*(K_{X_M}+\Delta _M)$.
Let $p:W\to X'$ and $q:W\to X'_M$ be a common resolution, then $$(\mu\circ p) ^*(K_X+\Delta )=(\mu _M\circ q)^*(K_{X_M}+\Delta _M)+G=q^*\phi '_*(K_{X'}+\Delta ')+G$$ where $G$ is effective, $(\mu _M\circ q)$-exceptional and the support of $(\mu \circ p)_*G$ 
is the set of $\phi$-exceptional divisors.
But then 
$$p^*(K_{X'}+\Delta ')=p^*(\mu ^*(K_X+\Delta )+F)=q^*\phi '_*(K_{X'}+\Delta ')+p^*F +G.$$ Since $(\mu \circ p )(G)\subset {\mathbf B}(K_X+\Delta /U)$,  the support of $G$ does not contain any center on $W$ of any divisor in $\mathcal E$. 
It follows that $p^*F+G$ is $q$-exceptional and that 
$F+p_*G$ is supported precisely on the set of $\phi '$-exceptional divisors. In fact, the inclusion $\subset$ is immediate and the inclusion $\supset$ follows since any $\phi '$-exceptional divisor $E$ with $\mu (E)\subset {\mathbf B}(K_X+\Delta /U)$ satisfies $a(E;X,\Delta )<a(E;X_M,\Delta _M)$ and hence is contained in $p_*G$ whilst any $\phi '$-exceptional divisor $E$ with $\mu (E)\not \subset {\mathbf B}(K_X+\Delta /U)$ is contained in ${\rm Supp }(F)$ by the definition of $X'_M$. Thus $\phi'$ is a good minimal model for $K_{X'}+\Delta '$ over $U$.

We now assume that $(X',\Delta ')$ has a good minimal model over $U$ say $\phi ':X'\dasharrow X'_M$.
By \eqref{dltmmp1}, we may assume that $\phi'$ is given by a sequence of $(K_{X'}+\Delta ')$-flips and divisorial contractions over $U$.
Then the rest of the proof is exactly the same as the second paragraph of the proof of \eqref{dltmmp1}.
\end{proof}

\begin{lemma}\label{l-11} Let $f\colon X\to U$ be a projective morphism, $(X,\Delta)$ a dlt pair and $\phi :X\dasharrow X_M$ a good minimal model for $K_X+\Delta$ over $U$. If $\psi :X\dasharrow Y={\rm Proj} R(X/U,K_X+\Delta )$ is a morphism, then there is a good minimal model for $K_X+\Delta$ over $Y$.
\end{lemma}
\begin{proof} Let $\mu :X'\to X$ and $\Delta '$ be defined as in \eqref{l-birmm}, then it suffices to show that $K_{X'}+\Delta '$ has a good minimal model over $Y$. 
Thus, we may assume that $\nu:X'\dasharrow X_M$ is a morphism.
Recall that $K_{X'}+\Delta '=\mu ^*(X_X+\Delta)+E$ where $E$ is effective and $\mu$-exceptional and note that $\mu ^*(K_X+\Delta )=\nu ^*(K_{X_M}+\Delta _M)+F$ where $F$ is effective, $\nu$-exceptional and $\mu _* F$ is supported on the $\phi$-exceptional divisors. 
Thus $K_{X'}+\Delta '=\nu ^*(K_{X_M}+\Delta _M)+F+E$ where $E+F$ is $\nu$-exceptional. If  $E+F\ne 0$, then some component of its support is contained in ${\mathbf B}_-(K_{X'}+\Delta '/X_M)$. It is then easy to see that after finitely many steps of the $(K_{X'}+\Delta ')$-minimal model program with scaling over $X_M$ we obtain $\phi ':X'\dasharrow X'_M$, where   ${\bf B}_{-}(K_{X'_M}+\phi '_*\Delta '/X_M)=0$ and hence $\phi '_* (F+E)=0$. If $\mu _m:X'_M\to X_M$ is the induced morphism, then $ K_{X'_M}+\phi '_*\Delta '=\mu _m^*(K_{X_M}+\phi _*\Delta )$ is semiample over $U$ and hence over $Y$.
Thus $\phi'$ is the required $K_{X'}+\Delta '$ has a good minimal model over $Y$.
\end{proof}

The following result is a generalization of \cite[4.4]{Lai09}.
The proof follows ideas of Shokurov (cf. \cite[3.7]{Birkar11}; see also \cite{Gongyo11}).
\begin{thm}\label{klt}
 Let $f\colon X\to U$ be a surjective projective morphism and $(X,\Delta)$ a dlt pair 
such that \begin{enumerate}
\item for a very general point $u\in U$, the fiber $(X_{u},\Delta_{u}=\Delta |_{X_u})$ has a good minimal model, and
\item the ring
$R(X/U;K_X+\Delta)$ is finitely generated.\end{enumerate}
Then $(X,\Delta)$ has a good minimal model over $U$.
\end{thm}
\begin{proof} We may assume that $f$ has connected fibers and that $U$ is affine. 
By \eqref{l-birmm}, we may assume that $(X,\Delta )$ is log smooth and that there is a morphism $\phi :X\to Y:={\rm Proj}R(X/U,K_X+\Delta )$.
We follow the strategy of \cite{Lai09}. 

Write $K_X+\Delta \sim _{\Q,U} G=G^h + G^v$, where $G \ge 0$ and $G^v$ (resp. $G^h$) is the vertical (resp. horizontal) part with respect to $\phi$. 
Note that $R(X_u, K_{X_u}+\Delta _u)$ is finitely generated, $Y_u\cong {\rm Proj}R(X_u, K_{X_u}+\Delta _u)$ and $X_u\to Y_u$ is a morphism.
By \eqref{l-11}, after running a $(K_{X}+\Delta)$-minimal model program over $Y$ with scaling of an ample divisor say $\eta :X\dasharrow X'$, we may assume that the very
general fiber of $\phi':X'\to Y$ is a good minimal model for $K_{X_u}+\Delta _u$ (cf. \eqref{l-rel}). 
Moreover, we may assume that ${\bf B}_{-}(K_{X'}+\Delta '/Y)$ does not contain any divisorial component where $\Delta '$ is the strict transform of $\Delta$. 
As the very general fiber $X'_{y}$ of $\phi'$ has Kodaira dimension zero, we have $(\eta _*G)|_{X'_{y}}  \sim_{\mathbb{Q}} 0$ and hence ${\eta _*G}^h = 0$. In particular, we may assume $\eta _* G$ is $\phi'$-vertical. 

Following the arguments in the fourth paragraph of the proof of \cite[4.4]{Lai09} one sees that, for any codimension 1 point $P\in Y$, ${\rm Supp}( \eta _*G^v)$ does not contain  ${\rm Supp}({\phi'}^{-1} P)$. In particular, $\eta _*G^v$ is of insufficient fiber type over $Y$ (see \cite[2.9]{Lai09}). This implies that some component of $\eta _*G^{v}$ is contained in $ {\bf B}_-(K_{X'} + \Delta'/Y )$. Thus $\eta _*G^v=0$ by our assumption of $X'$. Therefore, $\eta$ is a good minimal model for $K_X+\Delta$ over $Y$ so that $K_{X'}+\Delta '\sim _{\Q,U}{\phi'} ^*M$ where $M$ is an ample divisor on $Y$ over $U$ and $\phi' :X'\to Y$ is the induced morphism. Thus $\eta: X\dasharrow X'$ is a good minimal model for $K_X+\Delta$ over $U$.
\end{proof}

\begin{corollary}\label{c-mmp}
Let $h:Z\to Y$ be an algebraic fiber space projective over $U$, $(Y,\Theta)$  a klt pair and $(Z,\Delta_Z)$ a dlt pair such that $h^*(K_Y+\Theta)\sim_{\mathbb{Q},U}K_Z+\Delta_Z.$ Let $\eta: Y\dasharrow Y'$ be  a $(K_Y+\Theta)$-flip or  divisorial contraction over $U$. Then there exists a $(K_Z+\Delta_Z)$-negative birational contraction $\mu:Z\dasharrow Z'$ to a normal variety $Z'$ with a projective morphism $h':Z'\to Y'$ such that  $K_{Z'}+\Delta_{Z'}\sim _{\mathbb{Q},U}h'^*(K_{Y'}+\eta_*\Theta)$ where $\Delta_{Z'}=\mu_*\Delta_Z$. 

Furthermore, if $E$ is a $\mathbb{Q}$-Cartier $\mathbb{Q}$-divisor on $Y$, then $\mu_*h^*E\sim_{\mathbb{Q},U} h'^*\eta_*E.$ 
\end{corollary}
\begin{proof}If $\eta$ is a flip, let $Y\to W$ be the flipping contraction. Consider the morphism $(Z,\Delta_Z)\to W$. Since $Y\to W$ is birational, we know that $K_{Z_t}+\Delta_{Z_t}\sim_{\mathbb{Q}}0$ for a general fiber $Z_t$ of $Z\to W$. Since  for $d$ sufficiently divisible, 
$$R(Z/W,d(K_Z+\Delta))\cong  R (Y/W, d(K_Y+\Theta))$$ 
is a finitely generated $\mathcal{O}_W$-algebra,
by \eqref{klt}, there exists a good minimal model $Z'$ of $(Z,\Delta_Z)$ over $W$. Let $\Delta_{Z'}$ be the push forward of $\Delta_Z$ to $Z'$. Since $K_{Z'}+\Delta_{Z'}$ is semiample over $W$ and 
$$Y'={\rm Proj} R(Z/W,K_Z+\Delta_Z)= {\rm Proj}R (Z'/W, K_{Z'}+\Delta_{Z'}),$$
we have a morphism $h': Z'\to Y'$ such that $K_{Z'}+\Delta_{Z'}\sim_{\mathbb{Q},Y'} 0$. If $\eta$ is a divisorial contraction, we have $W=Y'$ and a same argument as above still holds. Finally, by \eqref{l-proj}, we conclude that
 $K_{Z'}+\Delta_{Z'}\sim _{\mathbb{Q},U}h'^*(K_{Y'}+\eta_*\Theta).$

 To see the last statement, since $E\sim_{\mathbb{Q},W}a(K_Y+\Theta)$ for some $a\in \mathbb{Q}$, then $h^*E\sim_{\mathbb{Q},W} a(K_{Z}+\Delta_Z)$, which implies 
 $\mu_*h^*E\sim_{\mathbb{Q},W}a(K_{Z'}+\Delta_{Z'}).$
 In particular, $\mu_*h^*E \sim_{\mathbb{Q},Y'} 0$, and we can conclude by \eqref{l-proj}. 
  \end{proof}


\section{Abundance for semi log canonical pairs}\label{s-4}

In this subsection, we prove the following statement. An immediate consequence is that if we assume that \eqref{t-lccl}$_{n-1}$ holds, then a similar result holds for the non-klt locus of  an $n$-dimensional pair which is usually not dlt but only dslt.
\begin{prop}\label{p-slc}
Let $(X,\Delta)$ be a  dslt pair, projective over a normal variety $U$ and $n: X^n\to X$ be the normalization. Write $n^*(K_{X}+\Delta)=K_{X^n}+\Delta^n+\Gamma$, where $\Gamma$ is the double locus. Assume that 
\begin{enumerate}
\item there exists an open set $U^0\subset U$, such that if we write $(X^0,\Delta^0)=(X,\Delta)\times_U U^0$, then $K_{X^0}+\Delta^0$ is semi-ample over $U^0$,
\item the image of any non-klt center of $(X^n,\Delta^n+\Gamma)$ intersects $U^0$, 
\item $K_{X^n}+\Delta^n+\Gamma$ is semi-ample over $U$.
\end{enumerate}
Then $K_{X}+\Delta$ is semi-ample over $U$.  
\end{prop}

The main technique we will use is Koll\'ar's powerful gluing theory. 
We will use the results of \cite{Kollar08} and \cite[Chapter 3]{Kollar10}
that give an inductive condition about when the quotient of  {a profinite equivalence relation} $R\rightrightarrows X$ exists. 
By \cite[Chapter 3]{Kollar10}, this inductive condition is well suited for 
the category of stratified varieties, where the stratification is given by non-klt centers, i.e.,  varieties with log canonical (lc) stratifications (cf. \eqref{e-lc}). 

The $f$-qlc stratification of a given minimal quasi log canonical (qlc) structure $f:(X,\Delta)\to Y$ is studied in \cite{KK10} (see also \cite{Ambro03} and \cite{Fujino08}).  
It turns out that the minimal qlc stratification shares many of the properties established for lc  stratifications. Our main observation in this section is that, Koll\'ar's gluing theory also works for compatible $f$-qlc stratifications.

\subsection{Koll\'ar's gluing theory}
In this subsection we briefly review Koll\'ar's theory of finite quotients.
We refer for \cite{Kollar08} and \cite[Chapter 3]{Kollar10} for more details.

\begin{defn}[{\cite[3.34]{Kollar10}}] Let $X$ be a scheme. A {\it stratification} of $X$ is a decomposition
of $X$ into a finite disjoint union of reduced locally closed
subschemes. We will consider stratifications where the strata are of
pure dimensions and indexed by the dimension. We write $X =
\cup_i S_iX$ where $S_iX \subset X$ is the $i$-th dimensional stratum.
Such a stratified scheme is denoted by $(X, S_*)$. We also assume
that $\cup_{ i\le j}S_iX$ is closed for every $j$. The {\it
boundary} of $(X, S_*)$ is the closed subscheme
$$BX := \cup_{i<\dim X}S_iX = X \setminus S_{\dim X}X.$$

Let $(X,S_*)$ and $(Y,S_*)$ be stratified schemes. We say that $f : X \to Y$ is a {\it stratified morphism} if $f(S_iX)\subset S_iY$ for every $i$, or equivalently, if $S_iX=f^{-1}(S_iY)$ for every $i$.
\end{defn}

\begin{defn}[{\cite[3.35]{Kollar10}}] 
Let $(X, S_*)$ be stratified variety. A relation $\sigma_i : R \rightrightarrows (X, S_*)$ is {\it stratified} if each $\sigma_i$ is stratifiable and $\sigma_1^{-1}S_* = \sigma_2^{-1}S_*$. Equivalently, there is a
stratification $(R, \sigma^{-1}S_*)$ such that $r \in \sigma^{-1}S_iR$ iff $\sigma_1(r) \in S_iX$ iff $\sigma_2(r) \in S_iX$.
\end{defn}

\begin{defn}[{\cite[3.37]{Kollar10}}] Let $X$ be an excellent scheme. We consider 4 normality
conditions on stratifications.

{\bf (N)} We say that $(X, S_*)$ has {\it normal strata}, or that it
satisfies condition (N), if each $S_iX$ is normal.

{\bf (SN)} We say that $(X, S_*)$ has {\it seminormal boundary}, or that
it satisfies condition (SN), if $X$ and the boundary
$BX=\cup_{i<\dim X}S_iX$ are both seminormal.

{\bf (HN)} We say that $(X, S_*)$ has {\it hereditarily normal strata}, or
that it satisfies condition (HN), if
\begin{enumerate}
\item[(a)]  $X$ satisfies (N),
\item[(b)] the normalization $\pi : X^n \to X$ is stratifiable, and
\item[(c)] its boundary $B(X^n)$ satisfies (HN).
\end{enumerate}

{\bf (HSN)} We say that $(X, S_*)$ has {\it hereditarily seminormal
boundary}, or that it satisfies condition (HSN), if
\begin{enumerate}
\item[(a)] $X$ satisfies (SN),
\item[(b)] the normalization $\pi : X^n \to X$ is stratifiable, and
\item[(c)] its boundary $B(X^n)$ satisfies (HSN).
\end{enumerate}
\end{defn}

The first example is the following.
\begin{ex}\label{e-lc}
(cf. \cite[3.47]{Kollar10})  Let $(X,\Delta)$ be a log canonical pair. Let $S^*_i(X,\Delta) \subset X$ be the union of all non-klt centers of $(X, \Delta)$ of dimension $\le i$, and $S_iX:=S^*_i(X,\Delta)\setminus S^*_{i-1} (X,\Delta).$
We call this {\it the log canonical stratification} or {\it lc stratification} of $(X, \Delta)$. By \cite[5.7]{KK10} the lc stratification $(X, S_*)$ satisfies all of the conditions (N), (SN), (HN), (HSN). Furthermore, if $D \subset \lfloor \Delta\rfloor$ is a divisor with normalization $D^n$, then $D^n \to X$ is a stratified morphism from the lc
stratification of $(D^n,{\rm Diff}^*_D \Delta)$ to the lc stratification of $(X,\Delta)$ where ${\rm Diff}^*_D\Delta$ is defined by $K_{D^n}+{\rm Diff}_D^*\Delta=(K_X+\Delta)|_{D^n}$.
\end{ex}

The following definitions generalize the concept of lc stratification.
\begin{defn}(\cite[5.3]{KK10}) Let $Y$ be a normal scheme. A {\it minimal quasi-log canonical structure}, or simply a {minimal qlc structure}, on $Y$ is a proper surjective morphism $ f : (X, \Delta) \to Y$ where
\begin{enumerate}
\item $(X, \Delta)$ is a log canonical pair, 
\item $\mathcal{O}_Y = f_* \mathcal{O}_X$ , and 
\item $K_X + \Delta \sim_{f,\mathbb{Q}} 0$.
\end{enumerate}
\end{defn}

\begin{defn}[{\cite[5.4]{KK10}}]\label{d-qlc} Let $Y$ be a normal scheme and assume that it admits a minimal qlc structure $ f : (X,\Delta) \to Y$. We define the {\it qlc stratification of $Y$ with respect to $f$} or simply {\it the $f$-qlc stratification} $(Y,S_*(X/Y,\Delta))$ in the following way: Let $\mathcal{H}_X$ denote the set of
all non-klt centers of $(X, \Delta)$, including the components of $\Delta$ and $X$ itself. For each $Z \in \mathcal{H}_X$	let
$$W_Z := f(Z) \setminus \bigcup_{Z' \in \mathcal{H}_X,  f ( Z ) \not\subset f (Z' )}f(Z').$$
Moreover, let $\mathcal{H}_{Y,f}=\{W_Z|Z\in \mathcal{H}_X \}$.  Then
$$Y=\coprod_{W \in \mathcal{H}_{Y,f}} W$$
is the {\it qlc stratification of $Y$ with respect to $f$}; its strata are the {\it $f$-qlc strata}. We can also define $S_i^*Y$ and $S_iY$ as above.
\end{defn}

We also refer to \cite[3.24-27]{Kollar10} for the definitions of {\it equivalence relation, pro-finite relation} and {\it geometric quotient}.

\begin{theorem}[{\cite[3.40]{Kollar10}}]\label{t-quotient}
 Let $(X, S_*)$ be an excellent scheme or algebraic space over a
field of characteristic 0 with a stratification. Assume that
$(X, S_{*})$ satisfies the conditions (HN) and (HSN). Let $R
\rightrightarrows X$ be a finite, set theoretic, stratified
equivalence relation. Then
\begin{enumerate}
\item the geometric quotient $X/R$ exists,
\item $\pi : X \to X/R$ is stratifiable and
\item $(X/R, \pi_*S_*)$ also satisfies the conditions (HN) and (HSN).
\end{enumerate}
\end{theorem}

The following lemma is important for our purposes.
\begin{lemma}\label{l-gl}
 Let $(X, S_*)$ be a stratified space satisfying $(N)$ and $Z \subset X$ a closed subspace which does not contain any of the irreducible components of the $S_iX$. Let $R\rightrightarrows (X, S_*)$ be a pro-finite, stratified set theoretic equivalence relation. Assume that $R|_{X\setminus Z}$ is a finite set theoretic equivalence relation. Then $R$ is also a finite set theoretic equivalence relation.
\end{lemma}

\begin{proof}See \cite[3.61]{Kollar10}.
\end{proof}

\subsection{Semi-ampleness for slc pairs}

Let $(X,\Delta )$ be a sdlt pair, $\Gamma^n$ be
the normalization of the double (non-normal) locus of $\Gamma \subset X^n$
and $\tau:\Gamma^n\to \Gamma^n$ be the induced involution. Then 
$(\tau_1,\tau_2):\Gamma^n\rightrightarrows X^n$ is a finite stratified equivalence relation and the
normalization map is given by the quotient morphism
$$\pi:X^n\to X=X^n/R,$$
where $R$ is the finite equivalence relation generated by $\Gamma^n$. 

If we assume that
$$L=\pi^*(K_{X}+\Delta)=K_{X^n}+\Delta^n+\Gamma$$ is
semi-ample on $X^n$, then we have an algebraic fibre space
$g^n:X^n\to Y^n$ given by $|mL|$ for $m>0$ sufficiently big and divisible. 
Let $h^n: \Gamma^n\to T^n$ be the
fibre space induced by $|mL|_{\Gamma^n}|$ after possibly replacing $m$ by a multiple. Then we have the commutative
diagram
\begin{diagram}
\Gamma^n&\pile{\rTo \\ \rTo}&X^n\\
\dTo^{h^n} & & \dTo_{g^n}\\
T^n &\pile{\rTo \\ \rTo} & Y^n,
\end{diagram}
where the morphisms $(\tau_1,\tau_2):\Gamma ^n\to X^n$ induce morphisms $(\sigma_1,\sigma_2):T^n\to Y^n$, where $\sigma_i\circ h^n$ indeed gives the Stein factorization of $g^n\circ \tau_i: \Gamma^n\to Y^n$. Therefore, $g^n:(X^n, \Delta^n+\Gamma)\to Y^n$ gives a   minimal  qlc structure which induces a minimal qlc stratification $S_*(X^n/Y^n,\Delta^n+\Gamma )$ on $Y^n$ (cf. \eqref{d-qlc}) and so does
$h^n:(\Gamma^n,\Theta)\to T^n$, where 
$$K_{\Gamma^n}+\Theta=(K_{X^n}+\Delta^n+\Gamma)|_{\Gamma^n}.$$ 

\begin{lemma}\label{l-HSN}
$Y^n$ and $T^n$ with the above minimal qlc structures satisfy 
conditions (HN) and (HSN).
\end{lemma}
\begin{proof}See \cite[Proposition 5.7]{KK10}.
\end{proof}

\begin{lemma}\label{l-sm}
$(\sigma_1,\sigma_2):T^n\rightrightarrows Y^n$ gives a stratified equivalence relation.
\end{lemma}
\begin{proof}To verify that $\sigma_1^{-1}S_*=\sigma_2^{-1}S_*$, we only need to check that $\sigma_j^{-1}S_*$ coincides with the minimal qlc stratification of $T^n $.
We apply induction on $i$ and assume that the statement is true for all dimensions less than $i$.

Since $\sigma_j$ is finite, we first verify that
$S_iT^n \subseteq \sigma_j^{-1} (S_iY^n)$  for all $i$,
which is equivalent to $\sigma_j(S_iT^n)\subseteq S_iY^n$ for all $i$. 
 Let $Z$ be an $i$-dimensional stratum of $T^n$ for the  minimal qlc stratification, it follows from the definition that there is a
non-klt center $Z_i$ of $(\Gamma^n, \Theta)$ such that $h^n(Z_i)=\bar{Z}$ which is the closure of $Z$. By inversion
of adjuction, $\tau_j(Z_i)$ is also a non-klt center of $(X^n,\Delta^n+\Gamma)$,
and so $\sigma_j(\bar{Z})= g^n\circ \tau_j(Z_i)$ is contained in $S^*_iY^n$. Therefore, it follows from the induction that 
$$\sigma_j({Z})=\sigma_j(\bar{Z}\setminus S^*_{i-1}(T^n))\subset S^*_iY^n\setminus S^*_{i-1}Y^n=S_iY^n.$$

For the other direction, let $Z$ be an $i$-dimensional stratum on $Y^n$. 
We assume $\sigma_j^{-1}(Z)\neq \emptyset $. Since $ \sigma_j(T^n)$ is a union of strata,  we know that $Z \subset \sigma_j (T^n)$. By definition, there is a non-klt center $Z_i$ of $(X^n,\Delta^n+\Gamma)$, such that $g^n(Z_i)=\bar{Z}$. 
Therefore, it follows from \cite[Theorem 1.7]{KK10} that if we let $p:\Gamma\bigcap (g^n)^{-1}(\bar Z)\to W$ and $q:W\to Z$ be the Stein factorization, then for every irreducible component $W_i\subset W$, there is a non-klt center $Z_{W_i}\subset \Gamma$ of $(X^n,\Delta^n+\Gamma)$ such that $p(Z_{W_i})=W_i$.  

It is easy to see that $\Gamma^n\to X^n$ is a stratified morphism for the lc stratification given by $(\Gamma^n,\Theta)$ and $(X^n,\Delta^n+\Gamma)$ (cf. \cite[3.48]{Kollar10}). So the preimages of $Z_{W_i}$ in $\Gamma^n$ are unions of non-klt centers  $\tilde{Z}_{W_{i,j}}$.
Thus $h^n(\tilde{Z}_{W_{i,j}})\setminus S^*_{i-1}T^n$ is a strata in $S_iT^n$and because of the induction assumption its image in $Y^n$ is $Z$. Therefore,  we conclude that $\sigma_j^{-1}(S_iY^n)\subseteq S_iT^n$.

\end{proof}
\begin{prop}\label{p-fer} $T^n\rightrightarrows Y^n$ generates a finite set
theoretic equivalence relation.
\end{prop}

\begin{proof}Let $Z\subset Y^n$ be the closed subset given by the preimage of $U\setminus U^0$ in $Y^n$. 
It suffices to verify the conditions of \eqref{l-gl}. By our assumptions  
there is no non-klt center of $(X,\Delta )$ contained in the preimage of $U\setminus U^0$ and hence no strata of $(Y^n,S_*)$ contained in $Z$. 
We let $X^0\to Y^0$ be the algebraic fibration induced by the semi-ample divisor $(K_{X}+\Delta)|_{X^0}$ over $U^0$.
Therefore, over $U^0$ the quotient $Y^n/T^n$ exists and is given by $Y^0$.
This implies that the restriction of the relation generated by $T^n$ to $Y^n\times _U U^0$ is finite.
\end{proof}

\begin{theorem}\label{gluing}
The quotient $Y$ of $T^n\rightrightarrows Y^n$ exists. Furthermore,
there exists a morphism $g:X\to Y$.
\end{theorem}
\begin{proof} By \eqref{l-HSN}, \eqref{l-sm}, \eqref{p-fer} and \eqref{t-quotient}.\end{proof}

\begin{proof}[Proof of \eqref{p-slc}] We follow very closely the arguments of \cite[3.84]{Kollar10}.  

Let $m$ be a positive integer such that $M:=m(K_{X^n}+\Delta^n+\Gamma)$ is Cartier and  base point free over $U$, 
inducing the algebraic fibre space $g^n:X^n\to Y^n$ over $U$. 
Therefore, there is a very ample line bundle $H$ on $Y^n$ such that 
${g^n}^*H=M$. Let $p_X:X^n_M\to X^n $ and $p_Y:Y^n_H\to Y^n$ be the total spaces of the line bundles $M$ and $H$. If we let $\Delta^n_M=p_X^{-1}(\Delta^n)$ and $\Gamma_M=p_X^{-1}(\Gamma )$,
then the morphism $g^n_M: (X^n_M,\Delta^n_M+\Gamma _M)\to Y^n_H$ yields a minimal qlc structure on $Y^n_H$, whose induced stratification is simply the pull 
back of the stratification of $Y^n$ coming from the minimal qlc structure on $(X^n,\Delta^n+\Gamma)$. In particular, this minimal qlc-stratification on $Y^n_H$ satisfies conditions (HN) and (HSN).

 Let $\Gamma^n_M\to \Gamma^n$ (resp. $T^n_H\to T^n$) be the total space of the restriction of the  line 
bundle $M$ (resp. $H$) to $\Gamma^n$ (resp. $T^n$). The involution $\tau$ of $\Gamma^n$ lifts to an involution of $\Gamma^n_M$ which preserves the non-klt centers of $(\Gamma^n_M, {\rm Diff}_{\Gamma^n_M}^*\Delta^n_M)$.
Similarly as above, this yields a pro-finite relation $({\sigma _1}_H,{\sigma _2}_H):T^n_H\rightrightarrows Y^n_H$ which is also stratified.
In what follows, we will check that the pro-finite relation $({\sigma_1}_H,{\sigma_2}_H):T^n_H\rightrightarrows Y^n_H$ indeed generates a finite set theoretic equivalence.

Since the morphism from each component of $R_H^n\times_{Y^n_H} S_iY^n_H$ to $S_iY^n_H$ is finite surjective,  where $R_H^n $ is the relation generated by $T^n_H\rightrightarrows Y^n_H$, to show the finiteness, i.e., to check that $R_H^n\times_{Y^n_H} S_iY^n_H\to S_iY^n_H$ is finite, we only need to verify this over the generic point $y$  of each strata of $\mathcal H _{Y^0,g^0}$ of $Y$ (see \eqref{l-gl}). The  pre-images $y_1,...,y_m$ of $y$ in $Y^n$ are all generic points of some strata of
$\mathcal H _{Y^n,g^n}$. 
Let $V_i$ denote the fiber of $Y^n_H$ over the point $y_i\in Y^n$, which is the 1-dimensional vector space 
$$V_i : = Y^n_H\otimes_{Y^n} k(y_i)\cong H^0(y_i, \mathcal O _{y_i}(H|_{y_i}))\cong H^0(X^n_{y_i}, \mathcal{O}_{X_{y_i}^n} (m(K _{X^n}+\Delta^n + \Gamma) |_{X^n_{y_i}} )),$$
where $X^n_{y_i}=X^n\times_{Y^n}\{y_i\}$. The involution $\sigma_H: T_H\to T_H$ gives a collection of isomorphisms $\sigma_{ijk} :V_i \to V_j$.  (A given $y_i$ can have several preimages in $T^n$, and each of these gives an isomorphism
of $V_i$ to some $V_j$. Hence there could be several isomorphisms from $V_i$ to $V_j$ for fixed $i, j$.) The $\sigma_{ijk}$ generate a groupoid. All possible composites
$$\sigma_{ij_2k_2} \circ \sigma_{j_2j_3k_3} \circ \cdots \circ \sigma_{j_nik_n} :V_i \to V_i$$
generate a subgroup of ${\rm Aut}(V_i)$, called the {\it stabilizer} $ {\rm stab}(V_i)$. Note that $T^n_H \rightrightarrows Y^n_H$ is a finite set theoretic equivalence relation iff ${\rm stab}(V_i) \subset {\rm Aut}(V_i) $ is a finite subgroup for every $i$.

From the definition of the minimal qlc  stratification, we know that there are non-klt centers of $(X^n,\Delta^n+\Gamma)$ which dominate the closure 
of $y_i$. Let $Z_i\subset X^n_{y_i}$ be the generic fiber over $y_i$ of a minimal one among all such non-klt centers. Then $Z_i$ is normal and if we write 
$(K_{X^n}+\Delta^n+\Gamma)|_{Z_i}=K_{Z_i}+\Delta_{Z_i}$, then $(Z_i,\Delta_{Z_i})$ is a klt pair and $m(K_{Z_i}+\Delta_{Z_i})\sim 0$. Therefore, we have 
\begin{enumerate}
\item  the Poincar\'e residue map gives a canonical isomorphism 
$$R^m : V_i \cong H^0(Z_i, \mathcal{O}_{Z_i}(m(K_{Z_i}+\Delta_{Z_i} ))),	\qquad \mbox{and}$$
\item  $\sigma_{ijk} : V_i \to V_j$ is induced by a birational map $\phi_{jik} : Z_j \dasharrow Z_i$ (cf. \cite[3.95]{Kollar10}).
\end{enumerate}
We note that by \cite[3.95]{Kollar10}, this is independent of the choice of $Z_i$ up to an element in the image of the {\bf B}-representation
${\rm Bir}(Z_i,\Delta_{Z_i})\to {\rm Aut}(V_i).$ 
 From this identification, ${\rm stab}(V_i)$ is contained in the image of the {\bf B}-representation. 
It follows from the argument in \cite[3.84]{Kollar10} or \cite[4.5]{Gongyo10} that the image of the {\bf B}-representation is finite.
Replacing $m$ by a multiple, we may assume that the image of the {\bf B}-representation is trivial.

It follows from \eqref{t-quotient} that we obtain a quotient $A$ of $T^n_H\rightrightarrows Y^n_H$ which is a line bundle on $Y$ whose pull-back to $Y^n$ is $\mathcal O _{Y^n}(H)$. 
Since $H$ is a divisor on $Y^n$ which is very ample over $U$, it follows that $A$ is ample over $U$. 
It also follows that  $g^*A=\mathcal O _X(m(K_X+\Delta))$, and hence $K_X+\Delta$ is semi-ample over $U$.
\end{proof}

\section{Base point free theorem}

In this section, we recall a version of Kawamata's theorem (cf.
\cite{Kawamata85}, \cite{Ambro05}, \cite{Fujino05}, \cite{Fujino09} and \cite{GF11}) on good minimal models.

\begin{theorem}[Base point free theorem]\label{abundance}Let $f:X\to U$ be a projective morphism and $(X,\Delta)$ a $\mathbb{Q}$-factorial dlt
pair. Assume that there exists an open subset $U^0\subset U$, such that
\begin{enumerate}
\item the image of any strata $S_I$ of $S=\lfloor \Delta \rfloor$  intersects $U^0$,
 \item $K_X+\Delta$ is nef and $(K_X+\Delta)|_{X^0}$ is semi-ample over $U^0$ where $X^0=X\times_U U^0$, and
\item for any component $S_i$ of $S$, $(K_X+\Delta)|_{S_i}$ is
semi-ample over $U$.
\end{enumerate}
Then $K_X+\Delta$  is semi-ample over $U$.
\end{theorem}
\begin{proof}By \eqref{p-slc}, $(K_X+\Delta)|_S$  is semi-ample over $U$. The theorem now follows from \cite[1.1]{Fujino05}.
\end{proof}


\section{Termination of flips}
In this section, we will prove \eqref{t-lccl}. We proceed by induction on $n=\dim X$. Therefore we may
assume \eqref{t-lccl} holds in dimension ${n-1}$. As we have seen (cf. \eqref{abundance}), the remaining problem is to show the existence of a minimal model. Here our strategy is similar to \cite[Section 5]{BCHM10}.  

For any proper birational morphism $\mu:X'\to X$, we write $\mu^*(K_X+\Delta)+F=K_{X'}+\Delta'$ where $F$, $\Delta'$ are effective divisors without common components. 
By \eqref{l-birmm} it suffices to prove the existence of a minimal model for $(X',\Delta')$. Thus in the following we may  replace $X$ by a higher model $X'$.

In what follows $\bullet ^0$ will denote restriction to $U^0$ so that for example $X^0=X\times_U U^0$ and $f^0=f|_{X^0}$.  Let 
$\bar{Y}^0={\bf \rm Proj}R({X^0}/U^0;K_{X^0}+\Delta^0)$.  Let $\bar Y\supset \bar Y^0$  be a compactification over $U$.  
\begin{lemma}\label{c}
There exist a smooth variety $Y$,  a birational morphism $Y\to \bar Y$, a morphism $h_X:X\to Y$ and a dlt log pair $(Z,\Delta _Z)$ which are projective over $U$ such that
\begin{enumerate} 
\item there is $\psi:X\dasharrow Z$ a good minimal model for $K_X+\Delta$ over $Y$ and a corresponding projective morphism $h: Z\to  Y$ such that $h_X=h\circ \psi$, 
\item there exists an effective $\mathbb{Q}$-divisor $B$ and a $g$-nef $\mathbb{Q}$-divisor class $J$ such that $K_{ Y}+B+J$ is big over $U$ and  for sufficiently divisible $m\ge 0$
$${h_X}_*\mathcal{O}_X(m(K_X+\Delta ))\cong \mathcal{O}_Y(m(K_Y+B+J)).$$ 
\item $K_{Z}+\psi_*\Delta\sim_{\mathbb Q}  h^*(K_{ Y}+B+J)$,
\item $(Y, B)$ is dlt,  log smooth, and all strata of $T=\lfloor  B\rfloor$ intersect $Y^0$,
\item there are effective $\mathbb Q$-divisors $C$, $\Sigma$ and $A$ on $Y$ such that $K_Y+B+J\sim _{\mathbb Q, U}C+\Sigma +A$, $A$ is ample over $U$, ${\rm Supp} (\Sigma )\subset {\rm Supp } (T)$, $(Y,B+\epsilon C)$ and $(X,\Delta+\epsilon h_X^*C)$ are dlt for some $\epsilon >0$. Moreover,  we can assume $(X,\Delta+\epsilon h_X^*C)$ and $(X,\Delta)$ have the same non-klt centers, and
\item for any $0<\epsilon \ll 1$ there is a klt pair $(Y,\Theta_{\epsilon})$ where $\Theta _\epsilon \sim _{\Q, U}B+J+\epsilon (A+C)$.
\end{enumerate}
\end{lemma}
\begin{proof}
By \eqref{t-can},  after possibly replacing $X$ by a higher model $X'$, we obtain  $(h_X:X\to Y,B,J)$ and a positive integer $d$ such that
\begin{enumerate}
\item[$\bullet$] $(Y,B)$ is a log smooth dlt pair, $K_{{Y}}+{B}+{J}$ is big over $U$ and $J$ is nef over $U$,
\item[$\bullet$] ${h_X}_*\mathcal{O}_X(dm(K_X+\Delta ))\cong \mathcal{O}_Y(dm(K_Y+B+J))$ for all $m\ge 0$.
\end{enumerate}

Since $K_{{Y}}+{B}+{J}$ is big over $U$, 
we may write $K_{{Y}}+{B}+{J}\sim _{\Q, U}C+\Sigma +A$ 
where $A$ is ample over $U$, $C$ and $\Sigma$ are effective and have no common components and 
${\rm Supp}(\Sigma )\subset  {\rm Supp}(T)$. 
Replacing $X$ and ${Y}$ by higher models, we may further assume that $(Y,{\rm Supp }(B+C))$ is log smooth. In particular, ${\rm Supp }(C)$ does not contain any non-klt centers of $(Y,B)$. By the definition of $B$, we know that $h_X^*C$ does not contain any non-klt center of $(X,\Delta)$. Thus, if $0<\epsilon\ll 1$, then $(X,\Delta+\epsilon h_X^*(C))$ is dlt and  has the same non-klt centers as $(X,\Delta)$.
We can further assume that $({Y}, {B}+\epsilon C)$ is dlt and $J+\epsilon A$ is ample over $U$ so that $B+J+\epsilon(C+A)\sim _{\Q, U}\Theta _\epsilon$ where $(Y,\Theta _\epsilon )$ is klt.

Since, by (\ref{t-can}.3), $$R(X/Y, d(K_X+\Delta+\epsilon h_X ^*(C+A)))\cong R(Y/Y, d(K_Y+B+J+\epsilon (C+A))),$$ 
$R(X/Y,K_X+\Delta+\epsilon h_X ^*(C+A))$ is finitely generated over $Y$ and so by \eqref{l-11} and \eqref{klt} there is a good minimal model  for $K_X+\Delta$ over $Y$ which we denote by $\psi: (X,\Delta )\dashrightarrow (Z,\Delta _Z)$.
We then have $K_Z+\Delta _Z \sim _{\mathbb Q, Y}0$ thus  $K_Z+\Delta _Z \sim _{\mathbb Q} h^*(K_Y+B+J)$.  
Properties (1-6) are easily seen to hold.
\end{proof}

In particular, (\ref{c}.6) implies that, for any $0<\epsilon \ll 1$, the ring  $R(Y/U, K_Y+B+J+\epsilon(A+C))$ is finitely generated by \cite{BCHM10}. 
We next construct a  good model of $(X,\Delta+\epsilon h_X^*(A+C))$ which, after restricting over $U^0$, yields a good model of $(X,\Delta)\times_U U^0$.
\begin{lemma}\label{l-start} 
With the above notation, if $0<\epsilon\ll 1$, then after replacing $X$ by a higher model,
\begin{enumerate}\item there is a $K_Y+B+J+\epsilon (C+A)$ good minimal model over $U$ say $Y\dasharrow Y'$, 
\item $K_X+\Delta+\epsilon h_X^*(A+C)$ has a minimal model over $U$, $Z' $ equipped with a morphism $h': Z'\to Y'$, and
\item $(K_{Y'}+B'+J'+t (C'+A'))|_{{Y'}^0}$ is semi-ample over $U^0$ for $0\leq t \leq \epsilon$.
\end{enumerate}
\end{lemma}
\begin{proof}

 Since $ K_Y+B+J+\epsilon(A+C)\sim_{\mathbb{Q},U}K_Y+\Theta _\epsilon$, $(Y, \Theta _\epsilon)$ is klt and $Y\to \bar Y$ is birational, by \eqref{t-bchm},
 there is a minimal model of $K_{Y}+{B}+{J}+\epsilon (C+A)$  over $\bar Y$ say $\eta:Y\dasharrow {Y}''$. 
By \eqref{c-mmp}, we know that there is  a birational contraction $\mu:Z\dasharrow Z''$ and a morphism $h'':Z''\to Y''$ such that 
$$\mu_*(K_Z+\Delta_Z+\epsilon(A+C))\sim_{\mathbb{Q},U} h''^*(\eta_*(K_Y+\Theta_{\epsilon})).$$ 
Since $\psi$ is $(K_X+\Delta )$-negative and $\mu$ is $(K_Z+\Delta_Z+\epsilon h^*(A+C))$-negative,  we see that if $0<\epsilon \ll 1$, $\mu \circ \psi$ is $(K_X+\Delta +\epsilon h_X^*(C+A))$-negative, then $X\dasharrow Z''$ is a good model for $K_X+\Delta +\epsilon h_X^*(C+A)$ over $Y''$.

We note that  $Z''$ is  a minimal model of $(Z, \Delta_Z+\epsilon h^*(C+A))$ over $\bar Y$. 
Since $\bar{Y}^0={\rm Proj}R(X^0/U^0;K_{X^0}+\Delta ^0)$,  applying \eqref{l-rel} to the good model $Z^0$ of $(X^0,\Delta^0)$ over $U^0$,  we obtain that for $0< t \ll \epsilon$, $K_{Z''^0}+\mu_*(\Delta^0_{Z}+t h^*(A^0+C^0))$ is nef over $U^0$, i.e., $Z''^0$  is indeed a minimal model of $(X, \Delta+t h^*(C+A))|_{X^0}$ over $U^0$. Hence $Y''^0$ is a minimal model of $(Y, \Theta _t\sim _{\Q, U}B+J+t(C+A))|_{U^0}$ over $U^0$ for any $0< t\ll \epsilon$. When $t=0$, as $Z''^0$ is a weak log canonical model of $(X,\Delta)|_{U_0}$, we also know that $\eta_*(K_{Y}+B+J)|_{Y''^0}$ is the pull back of a relatively ample divisor on $\bar{Y}^0$ over $U^0$. Thus replacing $\epsilon$ by a sufficiently small one, we may assume that $(K_{Y''}+\eta_*\Theta_t)|_{Y''^0}$ is nef for $0\le t \le \epsilon$.   

Running the $(K_{Y''}+\eta _*\Theta _{\epsilon})$-MMP with scaling of an ample divisor over $U$, we obtain a
$(K_Y+\Theta _{\epsilon})$-minimal model over $U$ say $Y\dasharrow Y'$. Since over $U^0$, $Y'^0$ is  isomorphic to $Y''^0$, this implies that
$(K_{Y'}+B'+J'+t(A'+C'))|_{Y'^0}$ is semiample for $0\le t\le \epsilon$. Now it is easy to see that $Y'$ satisfies (1) and (3).

We apply \eqref{c-mmp} to $Z''\to Y''$ and the sequence of flips and divisorial contractions $Y''\dasharrow Y'$ and we obtain a birational contraction $Z''\dasharrow Z'$ and a morphism $Z'\to Y'$. Then it is easy to see that $Z'$ satisfies (2).
\end{proof}

\begin{lemma}\label{c-0} Let $N_{\epsilon }:=N_\sigma (K_{ {Y}}+ {B}+ {J}+\epsilon (C+A)/U)$. We may assume that 
 $${\rm Supp}(N_{\epsilon '})={\rm Supp}(N_{\epsilon })\qquad {\rm for\ all} \ 0<\epsilon '\leq \epsilon.$$ 
 \end{lemma}
\begin{proof} Note that
if $\eta: {Y}\dasharrow  {Y}_{\epsilon'}$ is a minimal model of $K_{Y}+  {B}+ {J}+\epsilon '(C+A)$ over $U$, then 
the support of $N_{\epsilon '}$ equals the set of $\eta$-exceptional divisors.

Since ${\rm Supp}(N_{\epsilon '})\subset {\rm Supp}(N_\sigma (K_{ {Y}}+ {B}+ {J}/U)+C+A)$, we may find a 
decreasing sequence $\epsilon _j$ with $\lim \epsilon _j=0$ such that
 $N_{\epsilon _j}$ is fixed (independent of $j$). Replacing $\epsilon$ by $\epsilon _1$, we may assume that $N_{\epsilon _i}=N_\epsilon$.
 
 Suppose that  $\epsilon _i>\epsilon '>\epsilon _{i+1}$, then by convexity of $N_\sigma$, we have ${\rm Supp}(N_{\epsilon '} ) \subset {\rm Supp}(N_{\epsilon_i} +N_{\epsilon_{i+1}})={\rm Supp}(N_\epsilon)$.
Suppose that the above inclusion is strict for an infinite decreasing sequence $\epsilon '_i$ with limit $0$.
Passing to a subsequence, we may assume that there is a component $M$ of the support of $N_{\epsilon }$ which is not contained in $N_{\epsilon '_i}$. But then we may find $\epsilon '_i>\epsilon _j >\epsilon '_k$ 
and by convexity of $N_\sigma$,  $M$ is not a component of ${\rm Supp}(N_{\epsilon _j})={\rm Supp}(N_\epsilon)$. This is a contradiction and so ${\rm Supp}(N_{\epsilon '} ) ={\rm Supp}(N_\epsilon)$.
\end{proof}

We now run a $(K_{{Y}'}+{B}'+{J}')$-MMP over $U$ with scaling of $\epsilon (C'+A')$ (see \eqref{r-mmp2} below). Then there is a decreasing sequence $\epsilon =s_0\geq s_1\geq s_2\geq \cdots >0$ such that 
$K_{{Y}_i'}+{B}_i'+{J}_i'+s(C'_i+A'_i)$ is nef over $U$ for all $s_i\geq s\geq s_{i+1}$ where $s_0=\epsilon$ and either the sequence is finite in which case we let $s_{N+1}=0$ or the sequence is infinite and we have $\lim s_i=0$. 
Here, for any divisor $G$ on $Y$, we let $G'$ (resp. $G'_i$) be the strict transform of $G$ on $Y'$ (resp. on $Y'_i$).
Note that by (\ref{l-start}.3), each rational map $Y'_i\dasharrow Y'_{i+1}$ restricts to an isomorphism over $U^0$. By the above lemma,  we know that for any $i>0$, the rational map $ {Y}'\dasharrow  {Y}'_{i}$ is an isomorphism in codimension $1$.

\begin {rmk}\label{r-mmp2} Since $( Y,  B+\epsilon (C+A))$ is dlt, $A$ is ample  over $U$ and $J$ is nef over $U$, for any $\epsilon> t>0$ we may pick an $\mathbb R$-divisor $\Theta _t \sim _{\mathbb R, U} B+ J+t(C+A)$ such that $( Y,\Theta _t+(\epsilon -t)(C+A))$ is klt. It follows that  $(Y',\Theta _t'+(\epsilon -t)(C'+A'))$ is also klt. 
We then fix a decreasing sequence $t_i$ with limit $0$ and we run the $(K_{ Y'}+\Theta _{t_1}')$-MMP over $U$ with scaling of $(\epsilon -t_1)(C'+A')$. In this way we obtain a $(K_{{Y}'}+{B}'+{J}'+{t_1}(C'+A'))$-MMP over $U$ with scaling of $(\epsilon-t_1) (C'+A')$ say $ Y'\dasharrow Y'_1\dasharrow \ldots \dasharrow  Y'_{i_1}$ where $s_{i_1}\geq t_1>s_{i_1+1}$. Note that $(Y'_{i_1},\Theta _{t_2}+(t_1 -t_2)(C'_{i_1}+A'_{i_1}))$ is also klt and therefore we may run the $(K_{Y'_{i_1}}+\Theta _{t_2})$-MMP over $U$ with scaling of $(t_1 -t_2)(C'_{i_1}+A'_{i_1})$. Repeating the above procedure we obtain the required MMP with scaling over $U$.
\end{rmk}
 
 The following proposition is the main technical result of this section.
 
 \begin{proposition}\label{p-term}
 This MMP with scaling terminates.\end{proposition}
\begin{proof} We assume that we have an infinite sequence as above and we deduce a contradiction.

Let $\Gamma _i\subset Y'_i$ be a flipping curve, then $(C_i'+A_i')\cdot \Gamma _i>0$ and $(K_{ {Y}_i'}+ {J}'_i+ {B}'_i)\cdot \Gamma _i<0$ so that $\Sigma'_i\cdot  \Gamma _i<0$. In particular, the flipping locus always intersects $T'_i=\lfloor   B _i'\rfloor$. 

\begin{lemma}\label{c-z} There exists a sequence of birational contractions $  Z \dasharrow   Z'=  Z '_0\dasharrow   Z_1'\dasharrow   Z _2'\dasharrow \ldots$ 
and morphisms $h_i:  Z'_i\to   Y'_i$ such that \begin{enumerate}
\item $K_{  Z '_i}+\Delta _{  Z _i'}\sim_{\Q} h_i^*(K_{  Y'_i}+  B '_i+  J'_i)$, 
\item $Z\dasharrow Z'_i$ is a good minimal model of $K_{  Z}+\Delta _{  Z}+  s_ih^*(C+A)$  over $U$, and 
\item $Z_i'\dasharrow   Z _{i+1}'$ is an isomorphism over $U^0$.
\end{enumerate} 
\end{lemma}
\begin{proof} Let $Z_0'=Z'$.
Proceeding by induction, we assume that we are given $h_i:(  Z_i',\Delta _{  Z _i'}+s_ih_i^*(C'_i+A'_i))\to   Y'_i$ as in the above claim.

We now apply \eqref{c-mmp} to $Z'_i\to Y'_i$ and obtain $Z'_{i+1}\to Y'_{i+1}$ which satisfies (1) and (2). Since $Z'_{i+1}$ in \eqref{c-mmp} can be assumed to be obtained by running MMP with scaling (cf \eqref{dltmmp1}) over $W$, we know that (3) also holds.  
\end{proof}

For all $i\gg 0$, we may assume that $Z_i'\dasharrow Z_{i+1}'$ are isomorphisms in codimension $1$.
Fix a component $G$ of $S=\lfloor \Delta \rfloor$ which is not contracted by $X\dasharrow Z_i$ and let $G_i'$ be its strict transform on $Z'_i$. 
Denote by $\Delta_{G '_i}={\rm Diff}_{G '_i}(\Delta_{  Z _i'}-G'_i)$. 
By a discrepancy computation (cf. the arguments in the proofs of Step 1 and Step 2 of \cite[4.2.1]{Corti07}), we may assume that the induced rational maps $G'_i\dasharrow G'_{i+1}$ are isomorphisms in codimension $1$ for any $i\ge i_0$.  Let $\nu _i:G_i^\nu\to G_i'$ be  a $\Q$-factoralization.
Therefore, if we let $\Phi_i=h_i^*(A'_i+C'_i)|_{G^{\nu}_i}$, then for any $i\geq i_0$, $G^{\nu}_i$ is a minimal model of $(G'_{i_0}, \Delta_{G'_{i_0}}+s_i\Phi_{i_0})$. In particular $N_\sigma (K_{G'_{i_0}}+\Delta_{G'_{i_0}}+s_i\Phi_{i_0}/U)=0$ for $i\ge i_0$.
By semicontinuity of $N_\sigma (\cdot /U)$, we have $N_\sigma (K_{G'_{i_0}}+\Delta_{G'_{i_0}}/U)=0$.

Since we are assuming \eqref{t-lccl}$_{n-1}$, we may assume that each $K_{G^\nu _{i_0}}+\Delta_{G ^\nu _{i_0}}=\nu _{i_0}^*(K_{G'_{i_0}}+\Delta_{G '_{i_0}})$ has a good minimal model $G^{\nu}_{i_0}\dasharrow\bar{G}$ over $U$ which is also a good minimal model over $U$ for $K_{G'_{i_0}}+\Delta_{G '_{i_0}}$ (and in fact for any $K_{G'_i}+\Delta_{G '_i}$ where $i\geq i_0$).  There is a morphism $\chi:\bar{G}\to \bar{D}={\rm Proj}(\bar{G}/U,K_{\bar{G}}+\Delta_{\bar{G}})$. Since $N_\sigma (K_{G'_{i_0}}+\Delta_{G'_{i_0}}/U)=0$,
it follows that $G^{\nu}_i\dasharrow \bar{G}$ is  an isomorphism in codimension 1 for $i\ge i_0$.  
Note that $(K_{G'_i}+\Delta_{G '_i})|_{G'_i\times _U U^0}$ is nef over $U^0$.
By  \eqref{dltmmp1}, we can assume that $G_i^\nu \dasharrow \bar G$ is obtained by a sequence of $(K_{G^{\nu}_i}+\Delta_{G^{\nu}_i})$-flips. 
In particular we can assume that $G_i^\nu \dasharrow \bar G$  is an isomorphism over $U^0$. 
 This implies that it is an isomorphism at the generic point of each non-klt center of $(\bar G,\Delta_{\bar{G}})$ (cf. \cite[3.10.11]{BCHM10}), where $\Delta_{\bar{G}}$ is the push-forward of $\Delta_{G_i'}$ on $\bar{G}$. Let $\bar{\Phi}$ be the push-forward of  $\Phi_i$ to $\bar{G}$. 
 Therefore, there exists a rational number $0<t_0 \ll  1$ such that $(\bar{G},\Delta_{\bar{G}}+t_0\bar\Phi)$ is dlt with the same non-klt centers as $(\bar{G},\Delta_{\bar{G}}).$ 

\begin{lemma}\label{l-rg} There exists a good minimal model $\bar{G}^m$ of $(\bar{G},\Delta_{\bar{G}}+t_0\bar{\Phi})$ over $\bar{D}$.
\end{lemma}
\begin{proof} By our assumption the images of the non-klt centers of $(\bar{G},\Delta_{\bar{G}})$ intersect $U^0$.
Thus all non-klt centers of $(\bar{G},\Delta_{\bar{G}}+t_0\bar{\Phi})$ intersect $\bar{G}^0=\bar{G}\times _U U^0$.
$(K_{{G'_i}}+\Delta_{{G'_i}}+t_0{\Phi}_i)$ is semi-ample over $U$ and hence so is $(K_{\bar{G}}+\Delta_{\bar{G}}+t_0{\bar{\Phi}})|_{\bar{G}^0}$ over $U^0$.
 Therefore, $(K_{\bar{G}}+\Delta_{\bar{G}}+t_0{\bar{\Phi}})|_{\bar{G}^0}$ is semi-ample over $\bar{D}^0$.
 The claim is now immediate from 
\eqref{t-lccl}$_{n-1}$.
\end{proof}
By \eqref{l-rel}, after possibly replacing $t_0$ by a smaller positive number, we can assume $\bar{G}\dasharrow \bar{G}^m$ is a minimal model of $(\bar{G},\Delta_{\bar{G}}+t_0\bar{\Phi})$  over $U$. By \eqref{t-lccl}$_{n-1}$ this minimal model is good.
It follows that $K_{\bar{G}^m}+\Delta_{\bar{G}^m}+t\bar{\Phi}^m$ is semi-ample over $U$
for all $0\leq t \leq t_0$ and so 
$$R(\bar{G}^m/U; K_{\bar{G}^m}+{\Delta}_{\bar{G}^m},K_{\bar{G}^m}+\Delta_{\bar{G}^m}+t_0\bar{\Phi}^m)$$
is finitely generated.
Since $\bar{G}\dasharrow \bar{G}^m$ is $K_{\bar{G}}+\Delta_{\bar{G}}+t\bar{\Phi}$ non-positive for $0\leq t \leq t_0$, 
$R(\bar{G}/U; K_{\bar{G}}+{\Delta}_{\bar{G}},K_{\bar{G}}+\Delta_{\bar{G}}+t_0\bar{\Phi})$
is also finitely generated.

\begin{lemma}\label{l-ismc} There exists an $i_0\gg 0$, such that the rational maps $Z'_i\dasharrow Z'_{i+1}$ are isomorphisms on a neighborhood of $G'_i$ for all $i\geq i_0$.
\end{lemma}
\begin{proof}
Let $\Gamma$ be a geometric valuation over $\bar G$ and $D$ be a $\Q$-divisor on $\bar G$ such that $\kappa (D/U)\geq 0$. Then we let $\sigma _\Gamma (D)={\rm inf}\{ {\rm mult}_\Gamma (D')|D\sim _{\Q,U} D'\geq 0 \}$. So we have
$$\sigma _\Gamma (K_{\bar G^m}+\Delta_{\bar G^m }+t\bar \Phi^m/U)=0\qquad \mbox{ for all } 0\leq t\leq t_0.$$
Note that $\bar G^m$ is isomorphic to $G^{\nu}_i $ in codimension 1 for any $i\gg 0$.
So if we assume $s_i<t_0$ for $i\ge i_0$, then $\sigma _\Gamma(\cdot /U) $ is linear on the cone $\mathcal C _i\subset {\rm Div }_{\R}(G^\nu _i)$ spanned by $\nu _i^*(K_{G'_i}+\Delta_{G '_i})$ and $\nu _i^*(K_{G'_i}+\Delta_{G '_i}+t_0 \Phi_i)$ (see e.g., \cite[5.2]{CL10}). 
Since $Z_i'\dasharrow Z_{i+1}'$ is a sequence of $(K_{Z_{i}'}+\Delta_{Z_{i}'})$-flips over $U$ which is
$(K_{Z_{i}'}+\Delta_{Z_{i}'} +s_ih_i^*(C_i+A_i))$-trivial, if $Z_i'\dasharrow Z_{i+1}'$ is not an isomorphism on a neighborhood of $G_i'$, then there exists a 
divisor $E$ over $G'_i$ such that for any rational number $0<\delta <s_i$, the discrepancies satisfy 
$$a(E;G'_i,K_{G'_i}+\Delta_{G '_i}+(s_i-\delta)\Phi_i)<a(E;G'_{i+1},K_{G'_{i+1}}+\Delta_{G '_{i+1}}+(s_i-\delta)\Phi_{i+1}).$$
This implies that $\sigma _E(K_{G'_i}+\Delta_{G '_i}+(s_i-\delta)\Phi_i/U)>0$. 
Since $\sigma _E(K_{G'_i}+\Delta_{G '_i}+s_i\Phi_i/U)=0$, this is a contradiction to the fact that
$\sigma _E(\cdot /U)$ is linear nonnegative on $\mathcal C _i$.
\end{proof}

Thus we may assume that each $  Z '_i\dasharrow   Z'_{i+1}$ is an isomorphism on a neighborhood of $\lfloor \Delta_{  Z _i'}\rfloor$ for $i\ge i_0$. 

Any component $\Xi_i$ of $\lfloor B_i' \rfloor$ is the birational image of a component $\Xi$ of $\lfloor B \rfloor$ on $Y$ which is not in $N_{\sigma}(K_Y+B+J+s_i(A+C)/U)$. By the calculation of Kawamata subadjunction (cf. \eqref{t-can}), we know that $\Xi$ is dominated by a component $G$ contained in $\lfloor \Delta_Z \rfloor$. Furthermore, since 
$$h^*(K_Y+B+J+s_i (A+C))\sim_{\mathbb{Q}} K_Z+\Delta_Z+s_i h^*(A+C),$$
we know that $G$ is not contained in $N_{ \sigma}(K_Z+\Delta_Z+s_i h^*(A+C)/U)$. This implies that it is not contracted by $Z\dasharrow Z_i'$. We denote it by $G_i'$.
Because the pull back of $K_{Y'_i}+B_i'+J_i'+(s_{i}-\delta)(A'_i+C'_i)$ to $G'_i $ is  $K_{G'_i}+\Delta_{G'_i}+(s_{i}-\delta)\Phi'_i$ which is nef for $G_i'\subset \lfloor \Delta_{Z_i'} \rfloor$ and any $\delta<s_{i}$,  then it follows that each $  Y'_i\dasharrow   Y'_{i+1}$  is an isomorphism on a neighborhood of $\Xi_i$. Since $\Xi_i$ could be any component of $\lfloor B_i \rfloor $, this is impossible and hence the given $(K_{  Y'}+  B'+ J')$-MMP with scaling over $U$ terminates. This finishes the proof of \eqref{p-term}.
\end{proof}
\begin{corollary}The rational map $  Z\dasharrow   Z'_N$ is $(K_{  Z}+\Delta _Z)$-non-positive and so $$R(X/U,K_X+\Delta )\cong R(Z'_N/U, K_{  Z'_N}+\Delta _{Z'_N}).$$\end{corollary}
\begin{proof} Immediate from the fact that $  Z\dasharrow   Z'_N$ is a minimal model of $K_{  Z}+\Delta_{  Z}+th^*(C+A)$  for all $0<t<s_N$.
\end{proof}
\begin{proof}[Proof of \eqref{t-lccl}]
By what we have seen above, $K_{Z'_N}+\Delta _{Z'_N}$ is nef over $U$ and in particular $K_{Z^0_N}+\Delta _{Z^0_N}$ is nef over $U^0$ where $Z^0_N=Z'_N\times _U U^0$. Notice moreover that 
$$R(Z^0_N/U^0,K_{Z^0_N}+\Delta _{Z^0_N})\cong R(X^0_N/U^0,K_{X^0}+\Delta ^0)$$ is finitely generated.
By \eqref{klt}, $K_{Z^0_N}+\Delta _{Z^0_N}$ has a good minimal model over $U^0$.
Thus, $K_{Z^0_N}+\Delta _{Z^0_N}$ is semi-ample over $U^0$.
By induction on the dimension $K_{T'_N}+\Delta _{T'_N}=(K_{Z'_N}+\Delta _{Z'_N})|_{T'_N}$ is semi-ample over $U$.
By \eqref{abundance}, $K_{Z'_N}+\Delta _{Z'_N}$ is semi-ample over $U$.
Thus 
$$R(X/U,K_X+\Delta )\cong R(Z'_N/U,K_{  Z'_N}+\Delta _{Z'_N}),$$ which is finitely generated and by \eqref{klt}, $(X,\Delta )$ has a good minimal model over $U$.
\end{proof}

\section{Proof of Corollaries}
\begin{proof}[Proof of \eqref{c-lcc}] It follows from the assumption that there exists a compactification $X^0\subset X^c$, such that $f^c:X^c\to U$ is projective. Let $\Delta^c$ be the closure of $\Delta^0$ in $X^c$. Let  $\pi: \bar{X} \to X^c$ be a log resolution of $(X^c,\Delta^c\cup {f^c}^{-1}(U\setminus U^0) )$. Write 
$$(\pi|_{\bar{X}^0})^*(K_{X^0}+\Delta^0)=K_{\bar{X}^0}+\bar{\Delta}^0-F^0, $$
where $\bar{X}^0=\pi^{-1}(X^0)$ and $\bar{\Delta}^0$ and $F^0$ are effective $\mathbb{Q}$-divisors without common components. There is a commutative diagram
\begin{diagram}
 \bar{X}^0  &\rTo   & \bar{X}&&  \\
    \dTo^{\pi|_{\bar{X}^0}}&    & \dTo^{\pi} &  & \\
X^0&\rTo&X^c&\rTo^{f^c}&U
\end{diagram}
Let us consider the pair $(\bar{X},\bar{\Delta})$, where $\bar{\Delta}$ is the closure of $\bar{\Delta}^0$ in $\bar{X}$. 
Thus the pair  $(\bar{X},\bar{\Delta})$ and the morphism $\bar X \to X^c$ satisfy the assumptions of \eqref{t-lccl} (where we choose the open set of $X^c$ to be $X^0$).  
Therefore, we can take $(X,\Delta)$ to be the relative log canonical model of $(\bar{X},\bar{\Delta})$ over $X^c$. Finally note that $(X,\Delta)\times _U U^0=(X^0,\Delta^0)$.
\end{proof} 

\begin{defn}\label{d-ss} Recall that a morphism $f:(X,\Delta)\to U$ from a lc pair to a smooth curve $U$ is {\it  semi-stable}, if for all $p\in U$, we have that $(X, {\rm Supp} (\Delta) +X_p )$ is log smooth and $X_p=f^{-1}(p)$ is reduced.
\end{defn}

\begin{proof}[Proof of \eqref{c-lcm}] Given $f^0:X^0 \to U$ an affine morphism of finite type, we can choose a closure $f: X^c\to U$ which is projective. 
It follows from \cite[7.17]{KM98} (which is essentially in \cite{KKMS73}) that after a surjective base change  $\theta:\tilde{U}\to U$, we can assume that there exists a log resolution $\tilde{\pi}: \bar{X}\to (\tilde{X}^c, \tilde{X}^c\setminus \tilde{X}^0)$ where $\tilde{X}^0=X^0\times_U \tilde{U}$ and $\tilde{X}^c$ is the normalization of the main component of $X^c\times_U \tilde{U}$, such that  $\bar{f}: (\bar{X}, {\rm Ex}(\tilde{\pi}) ) \to \tilde{U}$ is semi-stable.  

Denote by $\bar{X}^0=\tilde{\pi}^{-1}(\tilde{X}^0)$. We note that $\tilde{X}^0$ is  log canonical since $f^0$ is a lc morphism.  
\begin{diagram}
 \bar{X}^0  &\rTo   & \bar{X}&\rDashto&X  \\
    \dTo^{\tilde{\pi}|_{\bar{X}^0}}&    & \dTo^{\tilde{\pi}} & \ldTo^{\pi} & \dTo\\
\tilde{X}^0&\rTo&\tilde{X}^c&\rTo^{\tilde{f}^c}&\tilde{U}&\rTo^{\theta}&U
\end{diagram}
Write $(\tilde{\pi}|_{\bar{X}^0})^*K_{ \tilde{X}^0}=K_{\bar{X}^0}+E^0-F^0$, and take the closure of $E^0$ in $\bar{X}$ to be $\bar{\Delta}$. 
In particular, $(\bar{X},\bar{\Delta })$ is a family of semi-stable pairs over $\tilde{U}$. 
Therefore, $(\bar{X},\bar{\Delta}+\bar{X}_p)$ is dlt for any $p\in \tilde{U}$. 
Applying \eqref{t-lccl} to  $(\bar{X},\bar{\Delta})$ over $\tilde{X}^c$, where we choose the open set of $\tilde{X}^c$ to be   $\tilde{X}^0$, we conclude that we have a log canonical model $X={\rm Proj}R(\bar{X}/\tilde{X}^c,K_{\bar{X}}+\bar{\Delta})$ and a morphism $\pi:X\to \tilde{X}^c$. It is easy to see that the induced morphism $X\to \tilde{U}$ is an lc morphism. Since $\tilde{X}^0$ is log canonical, $\pi$ is an isomorphism over $\tilde{X}^0$. 
\end{proof}

\begin{proof}[Proof of \eqref{c-pm}] The proof is similar to the the proof of \eqref{c-lcm}. After compactifying the family and applying  semi-stable reduction, 
we can construct a semi-stable family $(\bar{X},\bar{\Delta})$ over  $\tilde{U}$ (where $\tilde U\to U$ is a finite surjective morphism), such that $(X^0,\Delta^0)\times_{U}\tilde{U}$ is the log canonical model of $(\bar{X},\bar{\Delta})\times _U U^0$. By $\eqref{t-lccl}$, there exists a $(K_{\bar{X}}+\bar{\Delta})$-log canonical model $\bar X\dasharrow X$ over $\tilde{U}$. This gives the required morphism $f:X\to \tilde{U}$.  \end{proof} 

\begin{proof}[Proof of \eqref{c-tri}]
If $\Delta '$ meets the generic fiber of $f$, then $K_X+\Delta ''$ is not pseudo-effective over $U$. In this case the result follows from (\ref{t-bchm}.3). 
Therefore, we may assume that  $\Delta '$ does not meet the generic fiber of $f$ so that $K_X+\Delta ''$ is pseudo-effective over $U$.
By \eqref{klt}, it suffices to verify that $R(X/U,K_X+\Delta'')$ is finitely generated. 

Let $U^0=U\setminus f(\Delta')$. Let $\nu:X'\to X$ be a log resolution of $(X,\Delta ' +\Delta '')$. We write $\nu^*(K_X+\Delta'')+\Gamma=K_{X'}+\Delta''_{X'}$, where $\Delta_{X'}''$ and $\Gamma$ are effective and do not have common components.
 \begin{claim}\label{c-c} After replacing $X'$ by a higher model, we can assume that there exists pairs  $(Z,\Delta''_Z)$ and $(Y,B+J)$ and a commutative diagram
\begin{diagram}
X'&\rDashto^{\mu}&Z&\rTo^h &Y  \\
 &\rdTo_\nu & & & \dTo^g \\
 &&X &\rTo^f & U, 
\end{diagram}
such that
\begin{enumerate} 
\item $\mu:X'\dasharrow Z$ is a birational contraction, $\Delta''_Z=\mu_*\Delta''_{X'}$ and $\mu$ is $(K_{X'}+\Delta''_{X'})$-negative;
\item $h: Z\to  Y$ is projective morphism such that  $K_{Z}+\Delta ''_{ Z}\sim_{\mathbb Q}  h^*(K_{ Y}+B+J)$,
where $(Y, B)$ is dlt and $J$ is a nef $\mathbb{Q}$-line bundle, 
\item let $T=\lfloor B \rfloor$, then there are $\mathbb Q$-divisors $C$, $\Sigma$ and $A$ on $Y$ such that $K_Y+B+J\sim _{\mathbb Q, U}C+\Sigma +A$, $A$ is ample over $U$, ${\rm Supp} (\Sigma )\subset {\rm Supp } (T)$ and $(Y,B+\epsilon C)$, $(Z,\Delta_Z''+\epsilon h^*C)$ are dlt for some $\epsilon >0$ , and
\item for any $0<\epsilon \ll 1$, 
there is a $K_Y+B+J+\epsilon (C+A)$ good minimal model over $U$ say $Y\dasharrow Y'$.
\end{enumerate}
\end{claim}
\begin{proof} The argument is similar to the ones for \eqref{c} and \eqref{l-start}.
\end{proof}

Then by the arguments in Section 5, we can run a $(K_{Y'}+B'+J')$-MMP with scaling of $\epsilon (A'+ C')$  over $U$.  We obtain a sequence of flips and divisorial contractions
$$Y'=Y'_0\dasharrow Y'_1\dasharrow \cdots \dasharrow Y_n' \dasharrow\cdots ,$$ and a non-increasing sequence of rational numbers 
$\epsilon =s_0\geq s_1\geq s_2\geq \cdots >0$ such that 
$K_{{Y}_i'}+{B}_i'+{J}_i'+s(C'_i+A'_i)$ is nef over $U$ for all $s_i\geq s\geq s_{i+1}$ where $s_0=\epsilon$ and either the sequence is finite in which case we let $s_{N+1}=0$, or the sequence is infinite and we have $\lim s_i=0$. 

As in \eqref{c-z}, there exists a sequence of birational contractions $  Z \dasharrow   Z'=  Z '_0\dasharrow   Z_1'\dasharrow   Z _2'\dasharrow \ldots$ 
and morphisms $h_i:  Z'_i\to   Y'_i$ such that \begin{enumerate}
\item $Z\dasharrow Z'_i$ is a good minimal model of $K_{  Z}+\Delta'' _{  Z}+  s_ih^*(C+A)$ over $U$ and 
\item $K_{  Z '_i}+\Delta ''_{  Z _i'}\sim_{\Q,U} h_i^*(K_{  Y'_i}+  B '_i+  J'_i)$.
\end{enumerate} 

In particular, we can assume that for $i$ sufficiently big,  $Z\dasharrow Z_i$ contracts the birational image of any divisor $\Gamma$ which is contained in ${\bf B}_-(K_Z+\Delta''_Z/U)$. Therefore, if we let $\Delta'_{Z}=\mu_*\nu^*\Delta'$, then $K_{Z'_i}+\Delta'_{Z'_i}+\Delta''_{Z_i'}\sim_{\mathbb{Q},U} 0,$ and $(Z'_i, \Delta'_{Z'_i}+\Delta''_{Z'_i})$ is lc where $\Delta'_{Z'_i}$ and $\Delta''_{Z'_i}$ are the push forward of $\Delta_Z'$ and $\Delta''_Z$ to $Z'_i$.

The rest of the argument also closely follows Section 5. Fix a component $G$ of $S=\lfloor \Delta''\rfloor$ such that $G$ is not contained in  
$N_\sigma (K_X+\Delta''/U)$. If we denote by $G'_i$  the birational 
transform of $G$ on $Z_i'$,  $G_i^{\nu}\to G_i'$ a $\mathbb{Q}$-factorialization, $\Delta'_{G '_i}=\Delta'_{  Z _i'}|_{G'_i}$ and  $\Delta''_{G '_i}={\rm Diff}_{G '_i}(\Delta''_{  Z _i'}-G'_i)$, then by a discrepancy computation, we may assume that the induced rational maps $G'_i\dasharrow G'_{i+1}$ are isomorphisms in codimension $1$ for $i\ge i_0\ge 1$. 

Let $\Delta'_{G^{\nu}_i}$ and $\Delta''_{G^{\nu}_i}$ be the pull back of $\Delta'_{G'_i}$ and $\Delta''_{G'_i}$ to $G^{\nu}_i$. Then $K_{G_i^{\nu}}+\Delta '_{G_i^{\nu}}+\Delta''_{G_i^{\nu}}\sim_{\mathbb{Q},U}0$ and $(G^{\nu}_i,\Delta'_{G^{\nu}_i}+\Delta''_{G^{\nu}_i})$ is a log canonical pair.
Since we are assuming \eqref{c-tri}$_{n-1}$, we may assume that each $K_{G^{\nu}_i}+\Delta''_{G ^{\nu}_i}$ has a good minimal model $\bar{G}$ over $U$ which is obtained by running a sequence of MMP with scaling of an ample divisor.  Then $K_{\bar G}+\Delta _{\bar G}''$ is dlt and $K_{\bar G}+\Delta _{\bar G}'+\Delta _{\bar G}''\sim _{\Q , U}0$ where $\Delta _{\bar G}'$ and $\Delta'' _{\bar G}$ are the push forwards of  $\Delta'_{G^{\nu}_i}$ and $\Delta''_{G^{\nu}_i}$. 
As $N_{\sigma}(K_{G^{\nu}_i}+\Delta''_{G ^{\nu}_i}/U)=0$, we know $G^{\nu}_i\dasharrow \bar{G}$ is an isomorphism in codimension 1 for any $i\ge i_0$. Let  $\chi:\bar G \to \bar D={\rm Proj}R(\bar G /U, K_{\bar G}+\Delta _{\bar G}'')$ be the corresponding morphism over $U$. Thus proceeding as in the proof of \eqref{l-ismc}, it suffices to verify that $  Z '_i\dasharrow   Z'_{i+1}$ is an isomorphism on a neighborhood of $\lfloor \Delta''_{  Z _i'}\rfloor$ for $i \gg 0$. We first need to prove the following lemma.

\begin{lemma}\label{l-out}
For $s_i$ small enough,  $(\bar{G},\Delta''_{\bar{G}}+s_i\bar{\Phi})$ has a good minimal model over $\bar{D}$, where $\bar{\Phi}$ is the the push forward of $\Phi_i:=h_i^*(A'_i+C'_i)|_{G^{\nu}_i}$ to $\bar G$. 
\end{lemma}
\begin{proof} Denote by $\psi: \bar{G}\to V$ the Stein factorization of $\bar{G}\to U$. By choosing $s_i$ small enough, we can assume that $(\bar{G},\Delta''_{\bar{G}})$ and $(\bar{G},\Delta''_{\bar{G}}+s_i\bar{\Phi})$ have the same non-klt centers. We claim that for any non-klt center $Z$ of $(\bar{G},\Delta''_{\bar{G}}+s_i\bar{\Phi})$, $\chi(Z)$ is not contained in $\chi(\Delta'_{\bar{G}})$ (where $\chi:\bar G \to \bar D$ is defined above). 
In fact, because $Z$ is a non-klt center of $(\bar{G},\Delta''_{\bar{G}})$, as $(\bar{G},\Delta'_{\bar{G}}+\Delta''_{\bar{G}})$ is log canonical, $Z$ is not contained in ${\rm Supp}(\Delta'_{\bar{G}})$. Therefore we only need to show ${\rm Supp}(\Delta'_{\bar{G}})=\chi^{-1}\chi(\Delta'_{\bar{G}})$. But this immediately follows from the fact that $-\Delta'_{\bar{G}}\sim _{\Q , U}K_{\bar G}+\Delta '' _{\bar G}$ is nef over $\bar{D}$. By a similar argument  ${\rm Supp}(\Delta'_{\bar{G}})$ coincides with the pre-image of $\psi(\Delta'_{\bar{G}})\subset V$.

Over $V\setminus\psi(\Delta'_{\bar{G}})$, $K_{\bar{G}}+\Delta''_{\bar{G}}$ is $\mathbb{Q}$-linearly equivalent to $0$, therefore, the restriction of $ \bar{D}={\rm Proj}R(\bar{G}/U, K_{\bar{G}}+\Delta''_{\bar{G}})\to U$ induces an isomorphism of $\bar{D}\setminus\chi(\Delta'_{\bar{G}})$ to its image  in $V$. Because $G^{\nu}_i \dasharrow \bar{G}$ is an isomorphism in codimension 1 and $V\to U$ is finite, the fact that $R({G}^{\nu}_i/U,K_{G^{\nu}}+\Delta''_{G^{\nu}}+s_i\Phi_i)$ is a finitely generated $\mathcal{O}_U$-algebra implies $R(\bar{G}/V,K_{\bar{G}}+\Delta''_{\bar{G}}+s_i\bar{\Phi})$ is a finitely generated $\mathcal{O}_V$-algebra. Thus by \eqref{klt}, $(\bar{G},\Delta''_{\bar{G}}+s_i\bar{\Phi})$ has a good model over $\bar{D}\setminus\chi(\Delta'_{\bar{G}})$. As $\bar{D}\setminus\chi(\Delta'_{\bar{G}})$ intersects with the image of any non-klt center of  $(\bar{G},\Delta''_{\bar{G}}+s_i\bar{\Phi})$, the lemma follows from \eqref{t-lccl}.
\end{proof}

As the proof of \eqref{l-ismc} shows, it follows that each $  Z'_i\dasharrow   Z'_{i+1}$  is an isomorphism on a neighborhood of $\lfloor   \Delta ''_{Z'_i}\rfloor$ for $i\gg 0$.
Then by the same reason as the one in the last paragraph of the proof of \eqref{p-term}, this implies that the given $(K_{Y'}+  B'+ J')$-MMP with scaling over $U$ terminates.
It follows that $K_{Z'_N}+\Delta ''_{Z'_N}$ is nef over $U$ and the restriction of $K_{Z'_N}+\Delta''_{Z'_N} $ over $U^0$ is $\mathbb{Q}$-linear equivalent to 0. Since $K_{Z'_N}+\Delta''_{Z'_N}\sim_{\mathbb{Q}}-\Delta'_{Z'_N}$ is nef over $U$, it follows from the argument of \eqref{l-out} that the image of any non-klt center of $(Z'_N,\Delta''_{Z'_N})$ intersects $U^0$. 
By \eqref{t-lccl}, $K_{Z'_N}+\Delta ''_{Z'_N}$ is semi-ample over $U$. Since $X'\dasharrow Z'_N$ is $(K_{X'}+\Delta '' _{X'})$-non-positive, we have that
$$R(X/U, K_X+\Delta '')\cong R(X'/U, K_{X'}+\Delta ''_{X'})$$ is finitely generated.
By \eqref{klt} and \eqref{dltterm}, the proof is complete.
\end{proof}

\section{Remarks on the moduli functor}\label{s-func}
Over the last few decades, the study of the moduli functor of canonical polarized stable schemes has attracted a lot of interest. For simplicity, we restrict ourselves to the case when there is no boundary in the definition of the functor.  
\begin{defn}[Moduli of slc models, cf. {\cite[29]{Kollar10b}}]
Let $H(m)$ be an integer valued function. The moduli functor of semi log canonical models with Hilbert function $H$ is
\[
\mathcal{M}^{slc}_H(S)=\left\{
\begin{array}{cccccccc}
\mbox{ flat morphisms of proper families $X \to S$, }\\ 
\mbox{ fibers are  slc models with ample canonical class  	}\\
\mbox{and Hilbert function $H(m)$, $\omega_X$ is flat over $S$ }\\
\mbox{  and commutes  with base change.}\\	
\end{array}
\right\}
\]
\end{defn}
We refer the forthcoming book \cite{Kollar10} for a detailed discussion of this subject and to \cite{Kollar10b} for a more concise survey.
In this section, we will explain how to apply \eqref{c-pm}  together with  the arguments in \cite[Chapter 3]{Kollar10} to verify that the moduli functor $\mathcal{M}^{slc}_H$ satisfies the valuative criterion for  properness.

We assume that over a germ $C^0=C\setminus \{p\}$ of a smooth curve, we have a family of stable schemes $X^0\in \mathcal{M}^{slc}_H(C^0)$. 
Our goal is, to show that after an appropriate base change, we can compactify the family to get a family of stable schemes over $C$. Due to examples of Koll\'ar (cf. \cite[Proposition 1]{Kollar07b}), the pluricanonical ring of a slc variety is not necessarily finitely generated. Therefore, for general slc varieties, we can not directly run the minimal model program  to find the relative canonical model. 

Instead, Koll\'ar suggests the following strategy:
Consider the normalization, $n:\bar{X}^0=\sqcup_j {\bar{X}^0_j} \to X^0$ and write $n^*K_{X^0}=K_{\bar{X}^0}+\bar{\Gamma}^0$, where $\bar{\Gamma}^0=\sqcup_j \bar{\Gamma}^0_j$ denotes the double locus. Since $(\bar{X}^0,\bar{\Gamma}^0)$ comes from the normalization of $X^0$, we have the gluing data $(\bar{X}^0,\bar{\Gamma}^0,\tau^0)$ such that the quotient is $X^0$. 

\noindent{\bf Step 1:} Applying \eqref{c-pm}, we conclude that after a base change of $C$, there exist log canonical models  $(\bar{X}_j,\bar{{\Gamma}}_j)$ over $C$ which extend $(\bar{X}^0_j, \bar{\Gamma}^0_j)$ and admit a  projective lc morphism $\bar{X}_j\to C$. Let $\bar{X}=\sqcup \bar{X}_j$ and $\bar{\Gamma}$ the closure of $\bar{\Gamma}^0$ in $\bar{X}$.

\noindent{\bf Step 2:} We extend the gluing data $(\bar{X}^0,\bar{\Gamma}^0,\tau^0)$  to a  gluing data $(\bar{X},\bar{\Gamma},\tau)$. Indeed, we only need to show that we can extend the involution $\tau^0: \bar{\Gamma}^0\to \bar{\Gamma}^0$ to an involution $\tau: \bar{\Gamma}\to \bar{\Gamma}$, which is verified by the following lemma.
\begin{lemma} 
Let $(Y,\Gamma_Y)$ and $(Z,\Gamma_Z)$ be two projective lc morphisms (see \eqref{d-lc}) over a smooth curve $C$, such that $K_Y+\Gamma_Y$ and $K_Z+\Gamma_Z$ are ample over $C$.
We assume there exists a dense open set $C^0=C\setminus \{p\}$, such that the base changes $(Y^0,\Gamma^0_Y)$ and $(Z^0,\Gamma_Z^0)$ over $C^0$ are isomorphic. Then $(Y,\Gamma_Y)$ and $(Z,\Gamma_Z)$  are isomorphic.
\end{lemma}
\begin{proof} 
 Let $p:W\to Y$ and $q: W\to Z$ be a common resolution of $Y$ and $Z$. Write $p^*(K_Y+\Gamma_Y)+E_Y=K_W+\Phi_Y$ and $q^*(K_Z+\Gamma_Z)+E_Z=K_W+\Phi_Z$ such that $E_Y$ and $\Phi_Y$ (resp. $E_Z$ and $\Phi_Z$) are effective $\mathbb{Q}$-divisors without common components. Since $(Y,\Gamma_Y)\to C$ is an lc morphism, any divisor $E$ whose center on $Y$ is contained in a fiber has discrepancy $a(E, Y,\Gamma)\ge 0$ (cf. \cite[7.2(2)]{KM98}). Therefore,  $p_*^{-1}(\Gamma_Y)=\Phi_Y=\Phi_Z=q_*^{-1}(\Gamma_Z)$ and so
 $$Y\cong {\rm Proj} R(W/C, K_W+\Phi_Y)\cong {\rm Proj} R(W/C, K_W+\Phi_Z)\cong Z.$$  
\end{proof}

\noindent{\bf Step 3:} Applying the gluing theory developed in \cite[Chapter 3]{Kollar10}, Koll\'ar shows that the gluing data $(\bar{X},\bar{\Gamma},{\tau})$ yields a quotient $X$ over $C$, which is slc. Then $X/C$ is the corresponding family of canonical polarized slc models which extends $X^0/C^0$.

 \end{document}